\documentclass[a4paper,11pt]{article}

\usepackage{amsmath,amssymb,enumerate,graphicx,subfig,float,color,mathtools,amsthm,booktabs,multirow,array,xcolor,colortbl}
\usepackage{titlesec,bibentry}
\titlelabel{\thetitle.\quad}
\titleformat*{\section}{\Large\bfseries}
\titleformat*{\subsection}{\bfseries}
\definecolor{frame}{rgb}{0.9,0.9,0.9}
\definecolor{border}{rgb}{0.7,0.7,0.7}
\nobibliography*
\numberwithin{equation}{section}
\numberwithin{figure}{section}
\numberwithin{table}{section}
\newtheorem{theorem}{Theorem}[section]
\usepackage[linewidth=1pt,backgroundcolor=frame,linecolor=border]{mdframed}
\usepackage[numbers,sort&compress]{natbib}

\newmdtheoremenv[linewidth=1pt,backgroundcolor=frame,linecolor=border,innertopmargin=-0.1cm]{assumption}{Assumption}

\usepackage[margin=2cm]{geometry}

\usepackage{alltt}
\definecolor{mcg}{rgb}{0.1333,0.5451,0.1333}
\definecolor{msm}{rgb}{0.71,0.243,0.945}
\definecolor{mb}{rgb}{0.169,0.169,0.984}

\usepackage{hyperref}
\hypersetup{colorlinks = true, linktocpage = true, bookmarksopen = true, linkcolor = blue, urlcolor=blue, citecolor = blue, allcolors=blue}

\newcommand{\revision}[1]{{\color{black}{#1}}}
\newcommand{\change}[1]{#1}

\newif\ifsubset
\subsettrue 

\makeatletter
\newcommand{\redub}{}
\def\redub#1{%
  \@ifnextchar_%
    {\@redub{#1}}
    {\@latex@warning{Missing argument for \string\redub}\@redub{#1}_{}}%
}
\def\@redub#1_#2{%
    \underbracket[0.5pt]{\color{black}#1}_{\displaystyle\color{black} #2}%
}
\makeatother

\newmdtheoremenv[linewidth=1pt,backgroundcolor=frame,linecolor=border,ntheorem]{summary}{Summary.}

\graphicspath{{"./Figures/"}}

\begin{document}

\title{Finite volume schemes for multilayer diffusion}
\date{}
%

\author{Nathan G. March and Elliot J. Carr\footnote{Corresponding author: \href{mailto:elliot.carr@qut.edu.au}{elliot.carr@qut.edu.au}.}\\ School of Mathematical Sciences, Queensland University of Technology (QUT),\\ Brisbane, Australia.}
\maketitle

\begin{abstract} 
This paper focusses on finite volume schemes for solving multilayer diffusion problems. We develop a finite volume method that addresses a deficiency of recently proposed finite volume/difference methods, which consider only a limited number of interface conditions and do not carry out stability or convergence analysis. Our method also retains second-order accuracy in space while preserving the tridiagonal matrix structure of the classical single-layer discretisation. Stability and convergence analysis of the new finite volume method is presented for each of the three classical time discretisation methods: forward Euler, backward Euler and Crank-Nicolson. We prove that both the backward Euler and Crank-Nicolson schemes are always unconditionally stable. The key contribution of the work is the presentation of a set of sufficient stability conditions for the forward Euler scheme. Here, we find that to ensure stability of the forward Euler scheme it is not sufficient that the time step $\tau$ satisfies the classical constraint of $\tau\leq h_{i}^2/(2D_{i})$ in each layer (where $D_{i}$ is the diffusivity and $h_{i}$ is the grid spacing in the $i$th layer) as more restrictive conditions can arise due to the interface conditions. The paper concludes with some numerical examples that demonstrate application of the new finite volume method, with the results presented in excellent agreement with the theoretical analysis.
\end{abstract}

\section{Introduction}
Many industrial, environmental and biological problems involve diffusion processes across layered materials. For example, heat conduction in composites \citep{monte_2000,mikhailov_1983,mulholland_1972}, tumour growth across the white and grey matter components of the brain \citep{asvestas_2014,mantzavinos_2014}, contaminant transport across layered \change{soils} \citep{liu_1998,trefry_1999} and thermal conduction through skin layers during burning \cite{simpson_2017} all involve multilayer diffusion processes. Additionally, layered diffusion is of interest to the applied mathematics community as it can be thought of as a simple example of a multiscale problem when the number of layers is large \citep{carr_2016,carr_2017a}. These applications have led to a recent flourish in research activity focussed on analytical and numerical methods for solving mathematical models of multilayer diffusion \cite{carr_2016,carr_2016c,sheils_2016,rodrigo_2016,hickson_2011,kaoui_2018,gudnason_2018}.

This paper focusses on the numerical solution of the multilayer diffusion problem described as follows. Consider a diffusion process defined on an interval $[l_0,l_m]$ partitioned into $m$ distinct layers, such that $l_0<l_1<\hdots < l_{m-1} <l_m$, where $x=l_i$ ($i = 1,\hdots,m-1$) specifies the location of the interface between the $i$th and ($i+1$)th layers (see Figure \ref{fig:partition}). The resulting domain is denoted $[l_0,l_1,\hdots,l_{m-1},l_m$]. In this work, we define a linear diffusion equation on each layer together with general initial and boundary conditions:
\begin{gather}
\label{eq:PDE}
\frac{\partial u_i}{\partial t} = D_i \frac{\partial^2 u_i}{\partial x^2}, \quad l_{i-1} < x < l_i, \quad t>0,\\
\label{eq:initial_conditions}
u_i(x,0) = f_i(x),\\
\label{eq:left_BC}
a_Lu_1(l_0,t) -b_L\frac{\partial u_1}{\partial x}(l_0,t) = c_L,\\
\label{eq:right_BC}
a_Ru_m(l_m,t) +b_R\frac{\partial u_m}{\partial x}(l_m,t) = c_R,
\end{gather}
where $u_i(x,t)$ is the solution in the $i$th layer at position $x$ and time $t$, $D_i >0$ is the diffusion coefficient in the $i$th layer, $f_i(x)$ specifies the initial solution in the $i$th layer at position $x$, and $a_L$, $b_L$, $c_L$, $a_R$, $b_R$ and $c_R$ are non-negative constants satisfying $a_L + b_L >0$ and $a_R + b_R > 0$. We neglect the special case of Neumann conditions on both boundaries (i.e. $a_L = a_R = 0$).
\begin{figure}[t]
\centering
\includegraphics[width=16cm]{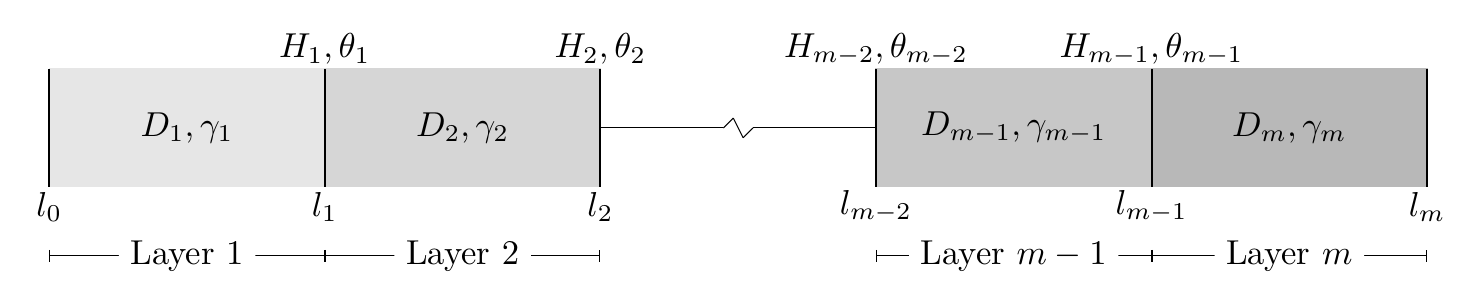}
\caption{Schematic diagram of a layered medium consisting of $m$ layers. The diffusion coefficient $D_i$ and conductivity coefficient $\gamma_i$ are constant in each layer ($i = 1,\hdots,m$) with the contact transfer and partition coefficients between the $i$th and $(i+1)$th layers denoted by $H_{i}$ and $\theta_i$ respectively.}
\label{fig:partition}
\end{figure}
A unique feature of multilayer problems are the internal boundary conditions that apply at the interfaces between adjacent layers. To close the problem (\ref{eq:PDE})--(\ref{eq:right_BC}), at each interface, $x = l_i$ ($i = 1,\hdots,m-1$), a pair of interface conditions \change{is} imposed, which we assume are chosen from one of the following four types:
\begin{enumerate}[(i)]
\item Type I: \begin{gather}
\label{eq:type1interface1}
u_i(l_i,t) =  u_{i+1}(l_i,t),\\
\label{eq:type1interface2}
D_i \frac{\partial u_i}{\partial x}(l_i,t) = D_{i+1} \frac{\partial u_{i+1}}{\partial x}(l_i,t).
\end{gather}
\item Type II: \begin{gather}
\label{eq:type2interface1}
D_i \frac{\partial u_i}{\partial x}(l_i,t) = H_i(u_{i+1}(l_i,t) - u_i(l_i,t)),\\
\label{eq:type2interface2}
D_{i+1} \frac{\partial u_{i+1}}{\partial x}(l_i,t) = H_i(u_{i+1}(l_i,t) - u_i(l_i,t)).
\end{gather}.
\item Type III: \begin{gather}
\label{eq:type3interface1}
u_i(l_i,t) = u_{i+1}(l_i,t),\\
\label{eq:type3interface2}
\gamma_{i} \frac{\partial u_{i}}{\partial x}(l_i,t) = \gamma_{i+1} \frac{\partial u_{i+1}}{\partial x}(l_i,t).
\end{gather}
\item Type IV:\begin{gather}
\label{eq:type4interface1}
u_i(l_i,t) = \theta_i u_{i+1}(l_i,t),\\
\label{eq:type4interface2}
D_{i} \frac{\partial u_{i}}{\partial x}(l_i,t) = D_{i+1} \frac{\partial u_{i+1}}{\partial x}(l_i,t).
\end{gather}
\end{enumerate}
Each of the four types of interface conditions (\ref{eq:type1interface1})--(\ref{eq:type4interface2}) model different physical processes at the interfaces and as a result find application to different industrial, environmental and biological problems. Both Type I and Type III conditions assume that the $i$th and $(i+1)$th layers are in perfect contact, that is, the solution is continuous across the $i$th interface. The difference is that equation (\ref{eq:type1interface2}) imposes continuity of the diffusive flux across the interface, whereas equation (\ref{eq:type3interface2}) allows for a more general formulation, where the flux depends on an arbitrary coefficient $\gamma_i > 0$ instead of the diffusion coefficient $D_i$. The latter interface condition is useful in heat conduction problems, for example, where the diffusion coefficient is the ratio of the thermal conductivity to the volumetric heat capacity, as equation (\ref{eq:type3interface2}) allows one to impose continuity of the heat flux as opposed to continuity of the diffusive flux \cite{hickson_2011}. For this reason, we will refer to $\gamma_{i}$ as the conductivity. In contrast, Type II and Type IV conditions give rise to imperfect contact at the $i$th interface, meaning that the solution is discontinuous across the $i$th interface. Equations (\ref{eq:type2interface1})--(\ref{eq:type2interface2}) specify that the flux is proportional, with proportionality coefficient (contact transfer coefficient) $H_i>0$, to the difference in the solutions $u_i(l_i,t)$ and $u_{i+1}(l_i,t)$ at the $i$th interface, whereas equation (\ref{eq:type4interface1}) specifies that $u_i(l_i,t)$ is proportional to $u_{i+1}(l_i,t)$, with proportionality coefficient (partition coefficient) $\theta_i > 0$. In the case of infinite contact transfer coefficient, $H_i \rightarrow \infty$, Type II conditions reduce to Type I. Type I conditions occur in pure diffusion problems \citep{carr_2016,hickson_2009a} and in modelling the growth of brain tumours \citep{mantzavinos_2014,asvestas_2014},  Type II conditions are used to model roughness/contact resistance between adjacent layers \citep{carr_2016,hickson_2009a}, Type III conditions appear in models of concentration diffusion in porous media \citep{sales_2002} and dissolved contaminant diffusion in aquitards \citep{liu_1998} and Type IV conditions ensure a discontinuity in the solution \change{that is useful in modelling drug release from microcapsules \citep{kaoui_2018,gudnason_2018}.}

Finite volume/difference schemes for the multilayer diffusion problem described above have been presented by \citet{carr_2016}, who implemented a finite volume scheme, and by \citet{hickson_2011}, who implemented a finite difference scheme. Both papers consider general Robin external boundary conditions and Euler time stepping schemes with the key difference lying in the treatment of the interface conditions. For Type I conditions (\ref{eq:type1interface1})--(\ref{eq:type1interface2}), the finite difference equation derived at the interface by \citet{hickson_2011} is equivalent to the finite volume equation derived at the interface by \citet{carr_2016}, provided that the grid spacing is identical in the two layers adjacent to the interface. However, for Type II conditions (\ref{eq:type2interface1})--(\ref{eq:type2interface2}), both \citet{carr_2016} and \citet{hickson_2011} propose two very different strategies. 

The approach taken by \citet{carr_2016} is to introduce two nodes at the interface to explicitly account for the discontinuity in the solution, with each node associated with one of two ``half'' finite volumes located either side of the interface. Under such a configuration, a finite volume \change{boundary} is located precisely at the interface, which allows the expressions for the flux specified by the interface conditions (\ref{eq:type2interface1})--(\ref{eq:type2interface2}) to be directly substituted into the finite volume equations. The remaining first-order spatial derivatives are discretised using a second-order central difference approximation giving rise to a finite volume scheme with a three-point stencil and a tridiagonal matrix structure. 

Conversely, \citet{hickson_2011} define two fictitious nodes, one slightly to the left and one slightly to the right of the interface. The interface conditions are incorporated into the formulation via the finite difference equations defined at the nodes immediately to the left and right of the interface. To approximate the second-order spatial derivatives at these two nodes a combination of forward, backward and central difference approximations are used that involve the two fictitious nodes at the interface. Due to their classification as fictitious, the solution at the two nodes at the interface must be expressed in terms of the solution at the surrounding nodes and this is achieved by solving an appropriate linear system of equations formulated using the interface conditions and suitable Taylor series expansions of the solution about neighbouring nodes. The net result is a finite difference scheme exhibiting an asymmetric four-point stencil, which erodes the tridiagonal structure of the matrix. 

In contrast to \citet{carr_2016}, \citet{hickson_2011} considered Type III conditions (\ref{eq:type3interface1})--(\ref{eq:type3interface2}) utilising a similar strategy to the one outlined above for Type II conditions (\ref{eq:type2interface1})--(\ref{eq:type2interface2}), with the exception being that a node is defined at the interface. However, as was the case for Type II conditions, the resulting matrix appearing in the finite difference discretisation is not tridiagonal. 

Other numerical schemes for multilayer diffusion can be found in papers by \citet{hein_2012}, \citet{mcginty_2015,mcginty_2016} and \citet{gudnason_2018}. \citet{hein_2012}, who considered only Type I conditions, included a finite difference equation for the flux condition  (\ref{eq:type1interface2}) directly, which was derived using second-order forward and backward difference approximations for the left and right-hand sides of equation (\ref{eq:type1interface2}), respectively. \citet{mcginty_2015,mcginty_2016} presented numerical solutions to a problem involving advection-diffusion across a layered medium, applying the strategy proposed by \citet{hickson_2011} for treating the interface conditions. \change{\citet{gudnason_2018} presented a finite element scheme that is suitable for Type IV conditions.} In each of these papers, as well as in the papers by \citet{carr_2016} and \citet{hickson_2011}, stability and convergence of the numerical schemes were not studied. 

This paper presents two main contributions:
\begin{enumerate}[(i)]
\item a new finite volume method for solving (\ref{eq:PDE})--(\ref{eq:right_BC}) capable of treating all four types of interface conditions \change{(\ref{eq:type1interface1})--(\ref{eq:type4interface2})}, and
\item stability and convergence analysis of the proposed finite volume method. 
\end{enumerate}
\noindent
The new finite volume method retains the tridiagonal matrix structure of Carr and Turner's \cite{carr_2016} scheme and is second-order accurate in space. A significant contribution of the work is the derivation of stability conditions for multilayer diffusion, which provide constraints on the time step and grid spacing that ensure stability of the finite volume scheme. \change{As we will see later,} a key finding is that to ensure stability when using the forward Euler scheme in time, it is not enough to simply enforce that the time step $\tau$ satisfies the classical stability constraint in each layer, i.e., $\tau\leq h_{i}^{2}/(2D_{i})$ for all $i = 1,\hdots,m$, where $D_{i}$ and $h_{i}$ are the diffusivity and uniform grid spacing in the $i$th layer, respectively.

The remaining sections of this paper are organised in the following way. In section \ref{sec:discretisation}, the new finite volume method is presented by outlining both the spatial and temporal discretisations employed.  In section \ref{sec:analysis}, the stability and convergence properties of the finite volume schemes are analysed and a complete list of the stability conditions is summarised. In section \ref{sec:numerical_experiments}, the finite volume method is applied to a series of test cases \change{with both the stability conditions and spatial accuracy of the finite volume schemes confirmed numerically}. The paper then concludes in section \ref{sec:conclusion} with a summary of the key findings of the work.

\section{Finite volume method}
\label{sec:discretisation}
We now derive a finite volume scheme for the discretisation of (\ref{eq:PDE})--(\ref{eq:right_BC}) capable of treating all four types of interface conditions (\ref{eq:type1interface1})--(\ref{eq:type4interface2}). To simplify the presentation, we note that each interface condition can be expressed in either of the following two general forms \citep{carr_2016c}:
\begin{enumerate}[(i)]
\item Type GI:
\begin{gather}
\label{eq:general_perfect1}
u_i(l_i,t) = \theta_iu_{i+1}(l_i,t), \quad t>0,\\
\label{eq:general_perfect2}
\gamma_{i} \frac{\partial u_{i}}{\partial x}(l_i,t) = \gamma_{i+1} \frac{\partial u_{i+1}}{\partial x}(l_i,t) , \quad t>0.
\end{gather}
\item Type GII:
\begin{gather}
\label{eq:general_imperfect1}
\gamma_i \frac{\partial u_i}{\partial x}(l_i,t) = H_i(\theta_iu_{i+1}(l_i,t) - u_i(l_i,t)), \quad t>0,\\
\label{eq:general_imperfect2}
\gamma_{i} \frac{\partial u_{i}}{\partial x}(l_i,t) = \gamma_{i+1} \frac{\partial u_{i+1}}{\partial x}(l_i,t) , \quad t>0.
\end{gather}
\end{enumerate}
\change{Under this formulation,} Type I conditions (\ref{eq:type1interface1})--(\ref{eq:type1interface2}), are recovered using Type GI by setting $\gamma_{i} = D_{i}$, $\gamma_{i+1} = D_{i+1}$ and $\theta_i = 1$ and Type II conditions (\ref{eq:type2interface1})--(\ref{eq:type2interface2}), are recovered using Type GII by setting $\gamma_{i} = D_{i}$ and $\gamma_{i+1} = D_{i+1}$. Type III conditions (\ref{eq:type3interface1})--(\ref{eq:type3interface2}) are recovered using Type GI by setting $\theta_i = 1$ while Type IV conditions (\ref{eq:type4interface1})--(\ref{eq:type4interface2}) can also be recovered using Type GI by setting $\gamma_i = D_i$ and $\gamma_{i+1} = D_{i+1}$. Even though in the limit that $H_i \rightarrow \infty$, Type GII reduces to Type GI, \change{we consider both general forms individually} in the finite volume method presented in this paper. As we will see in the next section, for Type GI, only one discrete unknown is required at the interface as there is an explicit relationship between $u_i(l_i,t)$ and $u_{i+1}(l_i,t)$ due to equation (\ref{eq:general_perfect1}), while for Type GII, two discrete unknowns are used to account for the discontinuity in the solution at the interface. 

\subsection{Spatial discretisation}
\label{sec:spatial_discretisation}
Spatial discretisation is carried out using a vertex-centered finite volume method. The mesh is assumed to be uniform in each layer, with nodes located at positions $x = x_{i,j} := l_{i-1}+jh_{i}$ for $j = 0,\hdots,n$ and $i = 1,\hdots,m$, where $h_{i} := (l_{i}-l_{i-1})/n$ is the grid spacing in layer $i$ and $n+1$ is the number of nodes in layer $i$\footnote{The finite volume scheme and analysis presented in this paper also holds if $n$ varies over the layers, however, we will ignore this for ease of notation.}. Configuring the mesh in this way defines two nodes at each interface, since $x_{i,n} = x_{i+1,0} = l_{i}$. The finite volume corresponding to the node at $x = x_{i,j}$, denoted by $V_{i,j}$, is defined as follows: $V_{i,j} := [x_{i,j}^w,x_{i,j}^e]$, where
\begin{align*}
x_{i,j}^w = \begin{cases} l_{i-1} & \text{if $j = 0$}\\
(x_{i,j-1}+x_{i,j})/2 & \text{else}
\end{cases}\qquad x_{i,j}^e = \begin{cases} l_{i} & \text{if $j = n$}\\
(x_{i,j}+x_{i,j+1})/2 & \text{else},
\end{cases}
\end{align*}
which defines ``half'' finite volumes either side of the interfaces and at the external boundaries.

Let $u_{i,j}(t)$ denote the discrete numerical approximation to $u_{i}(x_{i,j},t)$ and define $\Delta x_{i,j} := x_{i,j}^e - x_{i,j}^w$. Integrating equation (\ref{eq:PDE}) over an arbitrary finite volume $V_{i,j}$ and approximating $u_{i}(x,t)$ by $u_{i,j}(t)$ for all $x\in V_{i,j}$ yields: 
\begin{align}
\label{eq:general_fvm}
\Delta x_{i,j}\frac{du_{i,j}}{dt} = D_{i}\frac{\partial u_i}{\partial x} (x_{i,j}^e,t) - D_{i}\frac{\partial u_i}{\partial x} (x_{i,j}^w,t),
\end{align}
which is valid for all nodes in the mesh, that is, for all $i = 1,\hdots,m$ and $j = 0,\hdots,n$. The derivation of the finite volume equations using (\ref{eq:general_fvm}) depends on whether the finite volume $V_{i,j}$ corresponds to an interior node, boundary node or interface node. \change{The first two of these cases are quite standard, however, for completeness, we consider all three cases in the sections that follow}.
\subsubsection{Interior nodes}
Consider the equation (\ref{eq:general_fvm}) for nodes located in the interior of the layers, where $x_{i,j}\in(l_{i-1},l_{i})$. Using a second-order central difference approximation to the spatial derivatives, e.g.\begin{gather*}
\frac{\partial u_{i}}{\partial x}(x_{i,j}^{e},t) \approx \frac{u_{i,j+1}-u_{i,j}}{h_{i}},
\end{gather*}
noting that $\Delta x_{i,j} = h_{i}$ for interior nodes and rearranging, yields the \change{standard} finite volume equation:
\begin{align}
\label{eq:interior_ODEs}
\frac{du_{i,j}}{dt} = \frac{D_i}{h_i^2}(u_{i,j+1} - 2u_{i,j}+u_{i,j-1}).
\end{align}
The semi-discretised equation (\ref{eq:interior_ODEs}) applies for all interior nodes with the exception of a couple of special cases involving the nodes that are adjacent to the interfaces or external boundaries. These special cases are explained further in the sections that follow.
\subsubsection{Boundary nodes}
For the left external boundary node, located at $x = x_{1,0} = l_{0}$, equation (\ref{eq:general_fvm}) takes the form:
\begin{gather}
\label{eq:finite_volume2}
\frac{h_{1}}{2}\frac{du_{1,0}}{dt} = D_{1}\frac{\partial u_1}{\partial x}(x_{1,0}^{e},t) - D_{1}\frac{\partial u_1}{\partial x}(l_{0},t),
\end{gather}
since $\Delta x_{1,0} = h_{1}/2$ and $x_{1,0}^{w} = l_{0}$. To approximate the spatial derivative at $x = l_{0}$, the left boundary condition (\ref{eq:left_BC}) is rearranged and $u_{1}(l_{0},t)$ replaced by $u_{1,0}$ to give:
\begin{gather}
\label{eq:left_BC_flux}
\frac{\partial u_1}{\partial x}(l_0,t) \approx \frac{a_L u_{1,0} - c_L}{b_L}.
\end{gather}
Substituting (\ref{eq:left_BC_flux}) into (\ref{eq:finite_volume2}) and using a second-order central difference approximation to the remaining spatial derivative at $x = x_{1,0}^{e}$ appearing in (\ref{eq:finite_volume2}) produces the following finite volume equation:
\begin{gather}
\label{eq:left_BC_ODEs}
\frac{du_{1,0}}{dt} = -\frac{2D_1}{h_1}\left[\frac{1}{h_1}+\frac{a_L}{b_L}\right]u_{1,0} + \frac{2D_1}{h_1^2}u_{1,1}+\frac{2D_1c_L}{h_1b_L}.
\end{gather}
Note that equations (\ref{eq:left_BC_flux}) and (\ref{eq:left_BC_ODEs}) are valid only if $b_{L}\neq 0$. If $b_{L} = 0$\footnote{In this case, we must have that $a_{L}\neq 0$, otherwise the boundary condition (\ref{eq:left_BC}) vanishes.}, a finite volume equation is not included for $u_{1,0}$ as the discrete unknown can be computed directly from the boundary condition (\ref{eq:left_BC}) as follows: $u_{1,0} = c_{L}/a_{L}$. In this case, as $u_{1,0}$ is no longer treated as an unknown in the formulation, the finite volume equation for the node immediately to the right of the left boundary needs to be modified to eliminate $u_{1,0}$:
\begin{align}
\label{eq:left_BC_dirichlet_ODEs}
\frac{du_{1,1}}{dt} = \frac{D_1}{h_1^2}\left(u_{1,2} - 2u_{1,1}+\frac{c_L}{a_L}\right).
\end{align}
The right external boundary node, located at $x = x_{m,n} = l_{m}$, is treated in a similar manner to that described above for the left external boundary node. If $b_{R}\neq 0$, the following finite volume equation is derived:
\begin{align}
\label{eq:right_BC_ODEs}
\frac{du_{m,n}}{dt} =  \frac{2D_m}{h_m^2}u_{m,n-1} -\frac{2D_m}{h_m}\left[\frac{1}{h_m}+\frac{a_R}{b_R}\right]u_{m,n}+\frac{2D_mc_R}{h_mb_R}.
\end{align}
On the other hand, if $b_{R} = 0$ the discrete unknown is calculated directly from the boundary condition (\ref{eq:right_BC}) as $u_{m,n} = c_{R}/a_R$ and the finite volume equation for the node immediately to the left of the right boundary is modified to give:
\begin{align}
\label{eq:right_BC_dirichlet_ODEs}
\frac{du_{m,n-1}}{dt} = 
\frac{D_m}{h_m^2}\left(\frac{c_R}{a_R} - 2u_{m,n-1}+u_{m,n-2}\right).
\end{align}
\subsubsection{Interface nodes}
Consider the form of equation (\ref{eq:general_fvm}) for the two nodes positioned at $x_{i,n}$ and $x_{i+1,0}$, both of which are located at the $i$th interface ($x_{i,n}=x_{i+1,0}=l_{i}$). Since $\Delta x_{i,n} = h_{i}/2$ and $\Delta x_{i+1,0} = h_{i+1}/2$, we have:
\begin{gather}
\label{eq:interface1_ODEs}
\frac{h_{i}}{2}\frac{du_{i,n}}{dt} = D_{i} \frac{\partial u_i}{\partial x}(l_{i},t) - D_{i}\frac{\partial u_i}{\partial x}(x_{i,n}^{w},t),\\
\label{eq:interface2_ODEs}
\frac{h_{i+1}}{2}\frac{du_{i+1,0}}{dt} = D_{i+1}\frac{\partial u_{i+1}}{\partial x}(x_{i+1,0}^{e},t) - D_{i+1}\frac{\partial u_{i+1,0}}{\partial x}(l_{i},t).
\end{gather}
Multiplying equations (\ref{eq:interface1_ODEs}) and (\ref{eq:interface2_ODEs}) by $\gamma_i/D_i$ and $\gamma_{i+1}/D_{i+1}$, respectively, yields:
\begin{gather}
\label{eq:interface1_ODEs2}
\frac{\gamma_ih_i}{2D_i}\frac{du_{i,n}}{dt} = \gamma_i\frac{\partial u_i}{\partial x}(l_{i},t) - \gamma_i\frac{\partial u_i}{\partial x}(x_{i,n}^{w},t),\\
\label{eq:interface2_ODEs2}
\frac{\gamma_{i+1}h_{i+1}}{2D_{i+1}}\frac{du_{i+1,0}}{dt} =\gamma_{i+1}\frac{\partial u_{i+1}}{\partial x}(x_{i+1,0}^{e},t) - \gamma_{i+1}\frac{\partial u_{i+1,0}}{\partial x}(l_{i},t).
\end{gather}
The Type GI and Type GII interface conditions are now considered separately:\\

{\noindent\textit{Type GI conditions}}\\
Due to the interface condition (\ref{eq:general_perfect1}), \change{either  $u_{i,n}$ or $u_{i+1,0}$} can be eliminated from the finite volume formulation and equations (\ref{eq:interface1_ODEs2}) and (\ref{eq:interface2_ODEs2}) combined into a single finite volume equation at the interface. Adding equations (\ref{eq:interface1_ODEs2}) and (\ref{eq:interface2_ODEs2}) yields:
\begin{align}
\label{eq:interface_ODEs}
\frac{1}{2}\frac{d}{dt}\left(\frac{\gamma_ih_i}{D_i}u_{i,n}+\frac{\gamma_{i+1}h_{i+1}}{D_{i+1}}u_{i+1,0}\right) = \gamma_{i+1}\frac{\partial u_{i+1}}{\partial x}(x_{i+1,0}^{e},t) - \gamma_i \frac{\partial u_i}{\partial x}(x_{i,n}^{w},t),
\end{align}
by noting that the fluxes at $x = l_{i}$ cancel due to the interface condition (\ref{eq:general_perfect2}). Substituting $u_{i+1,0} = u_{i,n}/\theta_{i}$ and simplifying produces:
\begin{align}
\label{eq:interface_ODEs2}
\frac{du_{i,n}}{dt} = \frac{2D_iD_{i+1}\theta_i}{\gamma_ih_i\theta_iD_{i+1}+\gamma_{i+1}h_{i+1}D_i}\left[\gamma_{i+1}\frac{\partial u_{i+1}}{\partial x}(x_{i+1,0}^{e},t) - \gamma_i \frac{\partial u_i}{\partial x}(x_{i,n}^{w},t)\right].
\end{align}
Using second-order central difference approximations to the two remaining spatial derivatives in equation (\ref{eq:interface_ODEs2}) produces \change{the final semi-discretised finite volume equation}:
\begin{align}
\label{eq:interface_ODEs3}
\frac{du_{i,n}}{dt} = \frac{2D_iD_{i+1}\theta_i}{\gamma_ih_i\theta_iD_{i+1}+\gamma_{i+1}h_{i+1}D_i}\left[\frac{\gamma_i}{h_i}u_{i,n-1} - \left(\frac{\gamma_i}{h_i}+\frac{\gamma_{i+1}}{\theta_ih_{i+1}}\right)u_{i,n}+\frac{\gamma_{i+1}}{h_{i+1}}u_{i+1,1}\right].
\end{align}
As $u_{i+1,0}$ is no longer considered as a discrete unknown in the formulation, the finite volume equation for the interior node \change{immediately} to the right of the $i$th interface is modified to eliminate $u_{i+1,0}$ by making the substitution $u_{i+1,0} = u_{i,n}/\theta_{i}$:
\begin{align}
\label{eq:right_of_interface_ODEs}
\frac{du_{i+1,1}}{dt} = \frac{D_{i+1}}{h_{i+1}^2}\left(\frac{u_{i,n}}{\theta_i}- 2u_{i+1,1}+u_{i+1,2}\right).
\end{align}
\change{The finite volume equations (\ref{eq:interface_ODEs3}) and (\ref{eq:right_of_interface_ODEs}) are used in our finite volume method provided $\theta_{i}\geq 1$. If $\theta_{i}<1$ diagonal dominance of (\ref{eq:right_of_interface_ODEs}) is no longer maintained. This can be resolved, however, by eliminating $u_{i,n}$ instead of $u_{i+1,0}$ in the formulation. In this case, the finite volume equation (\ref{eq:interface_ODEs3}) is replaced with:
\begin{align}
\label{eq:reformulation}
\frac{du_{i+1,0}}{dt} = \frac{2D_iD_{i+1}}{\gamma_ih_i\theta_iD_{i+1}+\gamma_{i+1}h_{i+1}D_i}\left[\frac{\gamma_i}{h_i}u_{i,n-1} - \left(\frac{\theta_i\gamma_i}{h_i}+\frac{\gamma_{i+1}}{h_{i+1}}\right)u_{i+1,0}+\frac{\gamma_{i+1}}{h_{i+1}}u_{i+1,1}\right],
\end{align}
with the finite volume equation immediately to the left of the interface now modified to eliminate $u_{i,n}$ via the substitution $u_{i,n} = \theta_iu_{i+1,0}$:
\begin{align}
\label{eq:left_of_interface_ODEs}
\frac{du_{i,n-1}}{dt} = \frac{D_{i}}{h_{i}^2}\left(u_{i,n-2}- 2u_{i,n-1}+\theta_iu_{i+1,0}\right).
\end{align}
Note that (\ref{eq:left_of_interface_ODEs}) is diagonally dominant for $\theta_{i} < 1$. In summary, the finite volume equations (\ref{eq:interface_ODEs3}) and (\ref{eq:right_of_interface_ODEs}) are used if $\theta_i \geq 1$ whereas the finite volume equations (\ref{eq:reformulation}) and (\ref{eq:left_of_interface_ODEs}) are used if $\theta_i < 1$.}\\

\noindent{\textit{Type GII  conditions}}\\
In this case, both $u_{i,n}$ and $u_{i+1,0}$ are retained as discrete unknowns in the formulation as an explicit relationship does not exist between them. We make the following approximation:
\begin{align}
\label{eq:interface_flux_approximation}
\gamma_{i}\frac{\partial u_{i}}{\partial x}(l_{i},t) = \gamma_{i+1}\frac{\partial u_{i+1}}{\partial x}(l_{i},t) \approx H_{i}\left( \theta_iu_{i+1,0} - u_{i,n}\right),
\end{align}
making use of equations (\ref{eq:general_imperfect1}) and (\ref{eq:general_imperfect2}) and the approximations $u_{i}(x_{i,n},t)\approx u_{i,n}(t)$ and $u_{i}(x_{i+1,0},t)\approx u_{i+1,0}(t)$. Substituting (\ref{eq:interface_flux_approximation}) into equations (\ref{eq:interface1_ODEs2}) and (\ref{eq:interface2_ODEs2}), yields:
\begin{gather}
\label{eq:interface1_ODEs3}
\frac{\gamma_ih_i}{2D_i}\frac{du_{i,n}}{dt} = H_i(\theta_iu_{i+1,0} - u_{i,n}) - \gamma_i \frac{\partial u_i}{\partial x}(x_{i,n}^{w},t),\\
\label{eq:interface2_ODEs3}
\frac{\gamma_{i+1}h_{i+1}}{2D_{i+1}}\frac{du_{i+1,0}}{dt} = \gamma_{i+1}\frac{\partial u_{i+1}}{\partial x}(x_{i+1,0}^{e},t) - H_i(\theta_iu_{i+1,0} - u_{i,n}).
\end{gather}
Applying a second-order central difference approximation to the remaining partial derivatives in (\ref{eq:interface1_ODEs3}) and (\ref{eq:interface2_ODEs3}), we obtain the following pair of finite volume equations:
\begin{gather}
\label{eq:interface1_ODEs4}
\frac{du_{i,n}}{dt} = \frac{2D_i}{\gamma_ih_i}\left[\frac{\gamma_i}{h_i}u_{i,n-1} - \left(H_i+\frac{\gamma_i}{h_i}\right)u_{i,n}+\theta_iH_iu_{i+1,0}\right],\\
\label{eq:interface2_ODEs4}
\frac{du_{i+1,0}}{dt} = \frac{2D_{i+1}}{\gamma_{i+1}h_{i+1}}\left[H_i u_{i,n} - \left(\theta_iH_i+\frac{\gamma_{i+1}}{h_{i+1}}\right)u_{i+1,0}+\frac{\gamma_{i+1}}{h_{i+1}}u_{i+1,1}\right].
\end{gather}
\subsubsection{Summary of spatial discretisation}
Assembling the finite volume equations produces an initial value problem expressible in the following matrix form:
\begin{align}
\label{eq:system_ODEs}
\frac{d\mathbf{u}}{dt} = \mathbf{Au}+\mathbf{b},\quad\mathbf{u}(0) = \mathbf{u}^{(0)},
\end{align}
where $\mathbf{u} = [u_{1,0},u_{1,1},\hdots,u_{m,n-1},u_{m,n}]^T\in\mathbb{R}^{N}$ excluding $u_{1,0}$ if $b_L = 0$, \change{excluding $u_{m,n}$ if $b_R = 0$ and excluding either $u_{i+1,0}$ (if $\theta_{i}\geq 1$) or $u_{i,n}$ (if $\theta_{i}<1$)} for all interfaces $x = l_{i}$ at which Type GI interface conditions are applied. The initial solution vector $\mathbf{u}^{(0)}$ is calculated by evaluating the initial conditions (\ref{eq:initial_conditions}) at the nodes. The entries of $\mathbf{A}\in\mathbb{R}^{N\times N}$ and $\mathbf{b}\in\mathbb{R}^{N}$ are identified from the individual finite volume equations \change{(\ref{eq:interior_ODEs}), (\ref{eq:left_BC_ODEs})--(\ref{eq:right_BC_dirichlet_ODEs}), (\ref{eq:interface_ODEs3})--(\ref{eq:left_of_interface_ODEs}), (\ref{eq:interface1_ODEs4}) and (\ref{eq:interface2_ODEs4})}, with the number of unknowns $N = m(n+1)-q - r$, where $q$ is the number of interfaces at which Type GI interface conditions are applied ($q = 0, 1, \hdots,$ or $m-1$) and $r$ is the number of external boundary conditions of Dirichlet type\footnote{Note that $b_{L} = 0$ implies $a_{L}\neq 0$ and $b_{R} = 0$ implies $a_{R}\neq 0$ otherwise the boundary conditions (\ref{eq:left_BC}) and (\ref{eq:right_BC}) vanish.}:
\begin{gather*}
r = \begin{cases} 0 & \text{if $b_{L}\neq 0$ and $b_{R}\neq 0$,}\\
1 & \text{if $b_{L} = 0$ and $b_{R}\neq 0$ or $b_{L} \neq 0$ and $b_{R} = 0$,}\\
2 & \text{if $b_{L} = 0$ and $b_{R} = 0$.}
\end{cases}
\end{gather*}
An important observation is that all of the derived finite volume equations \change{(\ref{eq:interior_ODEs}), (\ref{eq:left_BC_ODEs})--(\ref{eq:right_BC_dirichlet_ODEs}), (\ref{eq:interface_ODEs3})--(\ref{eq:left_of_interface_ODEs}), (\ref{eq:interface1_ODEs4}) and (\ref{eq:interface2_ODEs4})} involve a three-point stencil, so $\mathbf{A}$ is tridiagonal.
\subsection{Temporal Discretisation}
\label{sec:temporal}
Temporal discretisation of the system of ODEs (\ref{eq:system_ODEs}) is carried out using one of three classical time stepping methods:
\begin{itemize}
\item Forward Euler: 
\begin{align}
\label{eq:forward_euler}
\mathbf{u}^{(k+1)} = \mathbf{A}_{F}\mathbf{u}^{(k)}+\tau\mathbf{b}, \quad \text{where $\mathbf{A}_{F} = \mathbf{I}+\tau\mathbf{A}$}.
\end{align} 
\item Backward Euler:
\begin{align}
\label{eq:backward_euler}
\mathbf{u}^{(k+1)} = \mathbf{A}_{B}(\mathbf{u}^{(k)}+\tau\mathbf{b}),\quad \text{where $\mathbf{A}_B = (\mathbf{I}-\tau\mathbf{A})^{-1}$}.
\end{align}
\item Crank-Nicolson: 
\begin{align}
\label{eq:crank_nicolson}
\mathbf{u}^{(k+1)} = \mathbf{A}_C\mathbf{u}^{(k)} + \tau\left(\mathbf{I}-\tfrac{\tau}{2}\mathbf{A}\right)^{-1}\mathbf{b}, \quad \text{where $\mathbf{A}_C = (\mathbf{I}-\tfrac{\tau}{2}\mathbf{A})^{-1}(\mathbf{I}+\tfrac{\tau}{2}\mathbf{A})$}.
\end{align}
\end{itemize}
In each of the three time discretisation schemes, $\tau$ is the (fixed) time step, $\mathbf{u}^{(k+1)}$ and $\mathbf{u}^{(k)}$ denote the numerical approximations to $\mathbf{u}(t)$ at $t = (k+1)\tau =: t_{k+1}$ and $t = k\tau =: t_{k}$, respectively, and $\mathbf{I}$ is the $N\times N$ identity matrix. \change{We remark here that the} inverses $(\mathbf{I}-\tau\mathbf{A})^{-1}$ and $(\mathbf{I}-\tfrac{\tau}{2}\mathbf{A})^{-1}$ exist as will be demonstrated in section \ref{sec:stability}. As $\mathbf{A}$ is tridiagonal, the linear system solves used in the backward Euler (\ref{eq:backward_euler}) and Crank-Nicolson (\ref{eq:crank_nicolson}) schemes can be carried out efficiently in $O(N)$ operations using the tridiagonal matrix algorithm.
\section{Theoretical analysis}
\label{sec:analysis}
\subsection{Stability}
\label{sec:stability}
A necessary and sufficient condition for stability of the forward Euler (\ref{eq:forward_euler}), backward Euler (\ref{eq:backward_euler}) and Crank-Nicolson (\ref{eq:crank_nicolson}) schemes is provided by the spectral radius of the iteration matrices \cite{sucec_1987}. Namely, the schemes are stable if and only if {
\begin{align}
\label{eq:spectral_radius}
\rho(\mathbf{A}_{i}) \leq 1,
\end{align}
for $i\in\{F,B,C\}$. Note that equation (\ref{eq:spectral_radius}) is the stability condition for asymptotic stability, that is, as the number of time steps $k\rightarrow\infty$. For finite $k$ (or equivalently a finite time interval $0\leq t \leq T$), (\ref{eq:spectral_radius}) may be weakened to $\rho(\mathbf{A}_{i})\leq 1 + O(\tau)$ \cite{lax_1956}, however, we take (\ref{eq:spectral_radius}) as our stability condition as it also ensures} that in the limit $k \rightarrow \infty$, the solution $\mathbf{u}^{(k)}$ approaches the steady state solution of the ODE system (\ref{eq:system_ODEs}):
\begin{align}
\label{eq:steady_state}
\mathbf{u^{(\infty)}} = {\lim_{k\rightarrow\infty} \mathbf{u}^{(k)} = } -\mathbf{A}^{-1}\mathbf{b}.
\end{align}
{For example, for the forward Euler method (\ref{eq:forward_euler}), we \change{have:}
\begin{gather}
\mathbf{u}^{(1)} = \mathbf{A}_{F}\mathbf{u}^{(0)} + \tau\mathbf{b},\quad \mathbf{u}^{(2)} = \mathbf{A}_{F}^{2}\mathbf{u}^{(0)} + \tau(\mathbf{A}_{F}\mathbf{b} + \mathbf{b}),
\end{gather}
and in general:}
\begin{align}
\label{eq:forward_euler_sum}
\mathbf{u}^{(k)} = \mathbf{A}_F^k\mathbf{u^{(0)}} + \tau \sum_{i = 0}^{k-1} \mathbf{A}_F^i\mathbf{b}.
\end{align}
Multiplying equation (\ref{eq:forward_euler_sum}) on the left by $\mathbf{A}_F$ and then subtracting  equation (\ref{eq:forward_euler_sum}) from this result yields:
\begin{align}
\label{eq:forward_steady}
(\mathbf{A}_F - \mathbf{I})\mathbf{u}^{(k)} = (\mathbf{A}_F-\mathbf{I})\mathbf{A}_F^k\mathbf{u^{(0)}}  + \tau(\mathbf{A}_F^k-\mathbf{I})\mathbf{b}.
\end{align}
If $\rho(\mathbf{A}_F) \leq 1$, then $\mathbf{A}_{F}^{k}\rightarrow\mathbf{0}$ as $k\rightarrow\infty$. Hence, taking the limit of (\ref{eq:forward_steady}) as $k\rightarrow\infty$, using the form of $\mathbf{A}_{F}$ (\ref{eq:forward_euler}) and solving for $\mathbf{u}^{(\infty)}$ yields the stated result (\ref{eq:steady_state}).

It follows from ({\ref{eq:spectral_radius}) that an expression for the largest (in magnitude) eigenvalue of the matrices $\mathbf{A}_F,\mathbf{A}_B$ and $\mathbf{A}_C$ can be used to determine stability conditions for each temporal discretisation method (\ref{eq:forward_euler})--(\ref{eq:crank_nicolson}). By considering the eigenvalues of the coefficient matrix $\mathbf{A}$ appearing in equation (\ref{eq:system_ODEs}), an expression for the eigenvalues of the matrices $\mathbf{A}_F,\mathbf{A}_B$ and $\mathbf{A}_C$ can be determined. If $\lambda$ is an eigenvalue of $\mathbf{A}$, then it is simple to show that:
\begin{itemize}
\item $1+\tau\lambda$ is an eigenvalue of $\mathbf{A}_F$.
\item $1/(1-\tau\lambda)$ is an eigenvalue of $\mathbf{A}_B$.
\item $(1+\tfrac{\tau}{2}\lambda)/(1-\tfrac{\tau}{2}\lambda)$ is an eigenvalue of $\mathbf{A}_C$. 
\end{itemize}
The conditions (\ref{eq:spectral_radius}) therefore give rise to the following well-known stability regions for the three time discretisation schemes:
\begin{itemize}
\item Forward Euler:\\
\begin{align}
\label{eq:AF_stability}
\rho(\mathbf{A}_F) \leq  1 \Leftrightarrow \max_{\lambda \in \sigma(\mathbf{A})} |1+\tau\lambda| \leq 1 \Leftrightarrow |\tau\lambda+1| \leq 1, \quad \forall \lambda \in \sigma(\mathbf{A}),
\end{align}
where $\sigma(\mathbf{A})$ refers to the spectrum of $\mathbf{A}$. Hence, for stability of the forward Euler scheme (\ref{eq:forward_euler}), we require that $\tau\lambda \in R_F$ for all $\lambda \in \sigma(\mathbf{A})$, where $R_F$ is the closed disc of radius \change{one} centred at $(-1,0)$ in the complex plane: $R_F = \{z\in \mathbb{C}:|z+1|\leq 1\}$ (see Figure \ref{fig:stability_regions}a).
\item Backward Euler:\\
\begin{align}
\label{eq:AB_stability} 
\rho(\mathbf{A}_B) \leq  1 \Leftrightarrow \max_{\lambda \in \sigma(\mathbf{A})}\frac{1}{ |1-\tau\lambda|} \leq 1 \Leftrightarrow |\tau\lambda-1| \geq 1, \quad \forall \lambda \in \sigma(\mathbf{A}).
\end{align}
Hence, for stability of the backward Euler scheme (\ref{eq:backward_euler}) we require that $\tau\lambda \in R_B$ for all $\lambda \in \sigma(\mathbf{A})$, where $R_B$ is the complement of the open disc of radius \change{one} centred at $(1,0)$ in the complex plane: $R_B = \{z\in \mathbb{C}:|z-1| \geq 1\}$ (see Figure \ref{fig:stability_regions}b).
\item Crank-Nicolson:\\
\begin{align}
\label{eq:AC_stability}
\rho(\mathbf{A}_C) \leq  1 \Leftrightarrow \max_{\lambda \in \sigma(\mathbf{A})}\left|\frac{1+\frac{\tau}{2}\lambda}{1-\frac{\tau}{2}\lambda}\right| \leq 1 \Leftrightarrow \Re(\lambda) \leq 0, \quad \forall \lambda \in \sigma(\mathbf{A}).
\end{align}
Hence, for stability of the Crank-Nicolson scheme (\ref{eq:crank_nicolson}), we require that $\tau\lambda \in R_C$ for all $\lambda \in \sigma(\mathbf{A})$, where $R_C$ is the closed-left half plane: $R_C = \{z\in \mathbb{C}:\Re(z)\leq 0\}$ (see Figure \ref{fig:stability_regions}c).
\end{itemize}

\begin{figure}[htbp]
\centering
\subfloat[Forward Euler]{\includegraphics[width=0.3\textwidth]{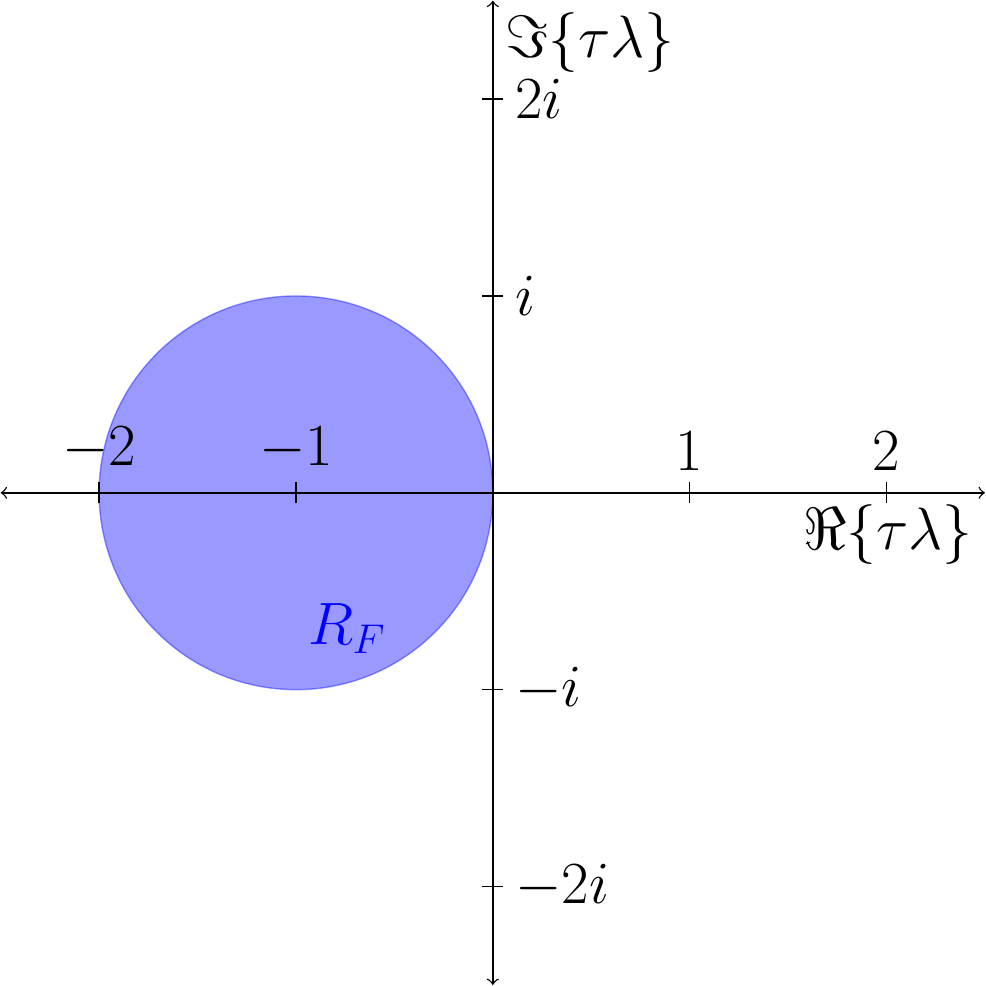}}\hspace{0.5cm}
\subfloat[Backward Euler]{\includegraphics[width=0.3\textwidth]{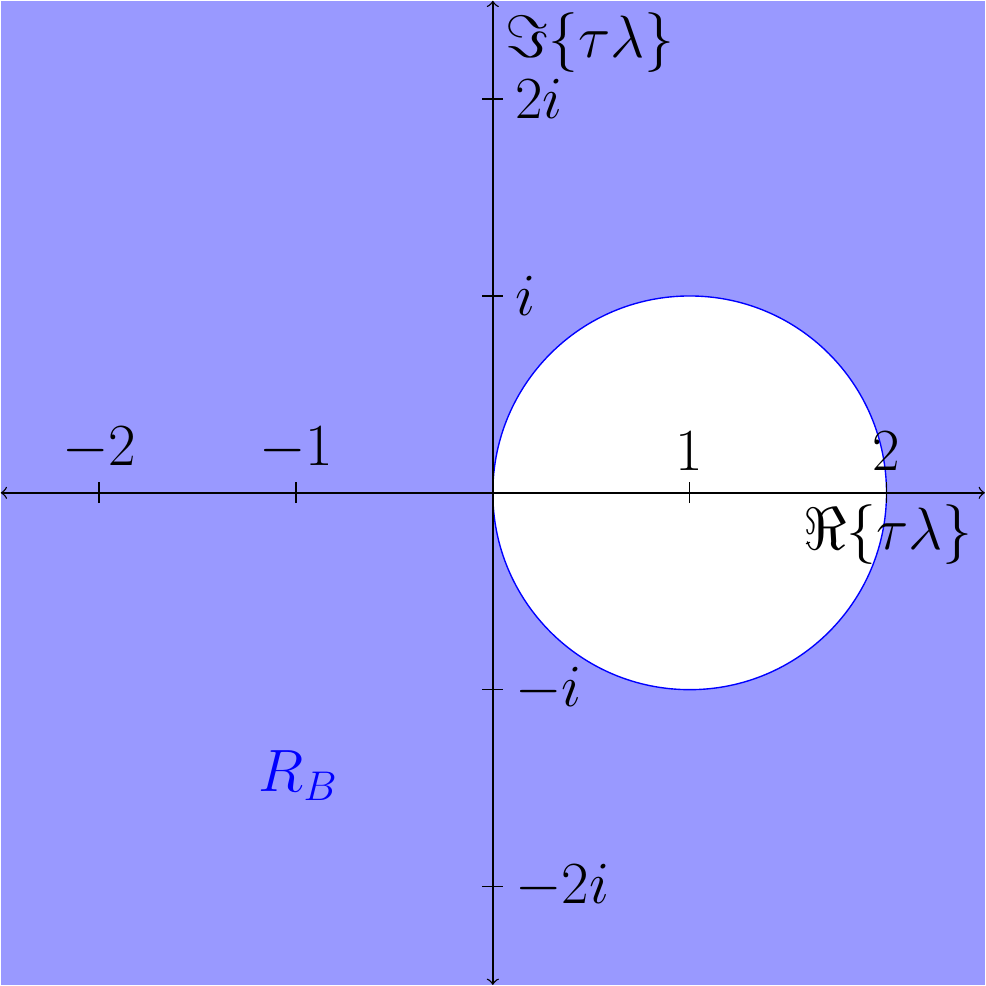}}\hspace{0.5cm}
\subfloat[Crank Nicolson]{\includegraphics[width=0.3\textwidth]{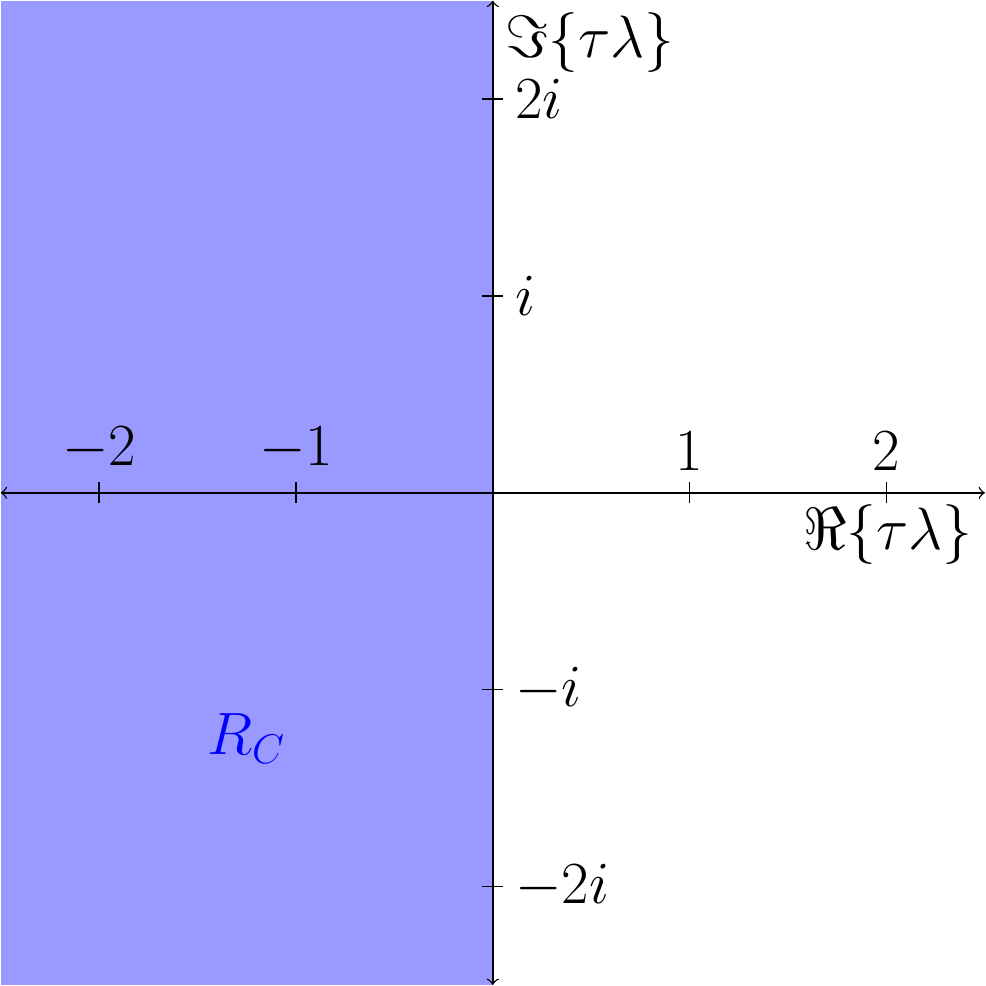}}
\caption{Stability regions for the three time discretisation schemes.}
\label{fig:stability_regions}
\end{figure}

The statements above provide constraints on the spectrum of the matrix $\tau\mathbf{A}$ that ensure stability of the forward Euler (\ref{eq:forward_euler}), backward Euler (\ref{eq:backward_euler}) and Crank-Nicolson (\ref{eq:crank_nicolson}) schemes. \change{We now prove the following result concerning the eigenvalues of $\mathbf{A}$.}

\begin{theorem}
All the eigenvalues of $\mathbf{A}$ are real and negative.
\label{thm:eigenvalues}
\end{theorem}
\begin{proof} \change{Recalling that the entries of the tridiagonal matrix $\mathbf{A}$ are identified from the individual finite volume equations (\ref{eq:interior_ODEs}), (\ref{eq:left_BC_ODEs})--(\ref{eq:right_BC_dirichlet_ODEs}), (\ref{eq:interface_ODEs3})--(\ref{eq:left_of_interface_ODEs}), (\ref{eq:interface1_ODEs4}) and (\ref{eq:interface2_ODEs4}), we see that $a_{p+1,p}a_{p,p+1}>0$ (i.e. the product of the sub-diagonal entry in column $p$, $a_{p+1,p}$, and the super-diagonal entry in row $p$, $a_{p,p+1}$, is positive) for all $p = 1,\hdots,N-1$. It follows that $\mathbf{A}$ is similar to the symmetric matrix $\mathbf{S} = \mathbf{DAD}^{-1}$, where $\mathbf{D} = \text{diag}(d_{1},\hdots,d_{N})$ with $d_1 = 1$ and $d_{p+1} = (a_{p,p+1}/a_{p+1,p})^{1/2}$ for all $p = 1,\hdots,N-1$ \citep{bernstein_2005}. Hence, all the eigenvalues of $\mathbf{A}$ are real.}

\change{To prove the eigenvalues of $\mathbf{A}$ are negative, we prove that the real eigenvalues of $\mathbf{S} = \mathbf{D}\widetilde{\mathbf{A}}\mathbf{D}^{-1}$ are positive, where $\widetilde{\mathbf{A}} := -\mathbf{A}$. As $\mathbf{S}$ is symmetric, all its eigenvalues are positive if and only if each of its principal minors are positive \citep{strang_2006}, that is, $\det(\mathbf{S}_k)>0$ for all $k = 1,\hdots,N$, where $\mathbf{S}_k$ is the matrix obtained by removing rows $k+1$ to $N$ and columns $k+1$ to $N$ from $\mathbf{S}$. For ease of explanation, in what follows, we consider the two-layer problem ($m = 2$) with Dirichlet boundary conditions ($b_L = b_R = 0$).}\\

\noindent\textit{Type GI conditions}\\
 \change{Consider Type GI conditions (\ref{eq:general_perfect1})--(\ref{eq:general_perfect2}) with $\theta_{i}\geq 1$ imposed at the interface. In this case,} 
\begin{align*}
\mathbf{S} = \begin{bmatrix}
\frac{2D_1}{h_1^2} & -\frac{D_1}{h_1^2} \\
-\frac{D_1}{h_1^2} & \frac{2D_1}{h_1^2} & -\frac{D_1}{h_1^2} \\
& \ddots & \ddots & \ddots \\
& & -\frac{D_1}{h_1^2} & \frac{2D_1}{h_1^2} & -\frac{D_1}{h_1^2} \\
& & & -\frac{D_1}{h_1^2} & \frac{2D_1}{h_1^2} & s_1\\
& & & &  s_1  & s_2 & s_3\\
& & & & & s_3 & \frac{2D_2}{h_2^2} & -\frac{D_2}{h_2^2}\\
& & & & & & -\frac{D_2}{h_2^2} & \frac{2D_2}{h_2^2} & -\frac{D_2}{h_2^2} \\
& & & & & & & \ddots & \ddots &\ddots \\
& & & & & & & & -\frac{D_2}{h_2^2} & \frac{2D_2}{h_2^2} & -\frac{D_2}{h_2^2} \\
& & & & & & & & & & \frac{2D_2}{h_2^2} & -\frac{D_2}{h_2^2} \\
\end{bmatrix},
\end{align*}
where
\begin{gather*}
s_1 =  \frac{-D_1\sqrt{2D_2\theta_1\gamma_1}}{h_1^{3/2}\sqrt{\gamma_1h_1\theta_1D_2+\gamma_2h_2D_1}},\quad s_2 = \frac{2D_1D_{2}\theta_1}{\gamma_1h_1\theta_1D_{2}+\gamma_{2}h_{2}D_1}\left[\frac{\gamma_1}{h_1}+\frac{\gamma_{2}}{\theta_1h_{2}}\right],\\
s_3 =  \frac{-D_2\sqrt{2D_1\gamma_2}}{h_2^{3/2}\sqrt{\gamma_1h_1\theta_1D_2+\gamma_2h_2D_1}}.
\end{gather*}
As $\mathbf{S}$ is \change{symmetric and tridiagonal}, the principal minors satisfy the recurrence relation:
\begin{align}
\label{eq:recurrence}
\det(\mathbf{S}_k) = d_k\det(\mathbf{S}_{k-1}) - e_k^2\det(\mathbf{S}_{k-2}),
\end{align}
where $d_k$ and $e_k$ are the diagonal and sub-diagonal elements, respectively, located in the $k$th row of $\mathbf{S}_k$. Applying the recurrence relation to each of the $N$ rows of $\mathbf{S}$ yields the following formulae for the principal minors of $\mathbf{S}$:
\begin{gather}
\label{eq:recurrence_first}
\det({\mathbf{S}}_{k}) = \frac{2D_{1}}{h_{1}^{2}}\det({\mathbf{S}}_{k-1}) - \frac{D_{1}^{2}}{h_{1}^{4}}\det({\mathbf{S}}_{k-2}), \quad k = 1,\hdots,n-1,\\
\label{eq:recurrence_n}
\det({\mathbf{S}}_{n}) = \frac{2D_{1}D_{2}(\gamma_{1}\theta_{1}h_{2}+\gamma_{2}h_{1})}{(\gamma_{1}h_{1}\theta_{1}D_{2}+\gamma_{2}h_{2}D_{1})h_{1}h_{2}}\det({\mathbf{S}}_{n-1}) - \frac{2D_{1}^{2}D_{2}\theta_{1}\gamma_{1}}{(\gamma_{1}h_{1}\theta_{1}D_{2}+\gamma_{2}h_{2}D_{1})h_{1}^3}\det({\mathbf{S}}_{n-2}),\\
\label{eq:recurrence_n+1}
\det({\mathbf{S}}_{n+1}) = \frac{2D_{2}}{h_{2}^{2}}\det({\mathbf{S}}_{n}) - \frac{2D_{1}D_{2}^2\gamma_{2}}{(\gamma_{1}h_{1}\theta_{1}D_{2}+\gamma_{2}h_{2}D_{1})h_{2}^{3}}\det({\mathbf{S}}_{n-1}),\\
\label{eq:recurrence_second}
\det({\mathbf{S}}_{k}) = \frac{2D_{2}}{h_{2}^{2}}\det({\mathbf{S}}_{k-1}) - \frac{D_{2}^{2}}{h_{2}^{4}}\det({\mathbf{S}}_{k-2}), \quad k = n+2,\hdots,N.
\end{gather}
With $\det({\mathbf{S}}_{1}) = 2D_{1}/h_{1}^{2}$ and $\det({\mathbf{S}}_{2}) = 3D_{1}^{2}/h_{1}^{4}$, the recurrence relation (\ref{eq:recurrence_first}) has solution
\begin{align}
\label{eq:recurrence_first_solution}
\det(\change{\mathbf{S}}_{k}) = (1+k)\left(\frac{D_{1}}{h_{1}^{2}}\right)^{k}, \quad k = 1,\hdots,n-1.
\end{align}
Substituting (\ref{eq:recurrence_first_solution}) into (\ref{eq:recurrence_n}) yields:
\begin{align}
\det({\mathbf{S}}_{n}) = \frac{2D_{1}^{n}D_{2}(\gamma_{1}\theta_{1}h_{2}+n\gamma_{2}h_{1})}{(\gamma_{1}h_{1}\theta_{1}D_{2}+\gamma_{2}h_{2}D_{1})h_{1}^{2n-1}h_{2}}.
\end{align}
Similar \change{analysis yields} the following results for the solutions of (\ref{eq:recurrence_n+1}) and (\ref{eq:recurrence_second}):
\begin{gather}
\det({\mathbf{S}}_{n+1}) = \frac{2D_{1}^{n}D_{2}^{2}[2\gamma_{1}\theta_{1}h_{2}+n\gamma_{2}h_{1}]}{(\gamma_{1}h_{1}\theta_{1}D_{2}+\gamma_{2}h_{2}D_{1})h_{1}^{2n-1}h_{2}^{3}},\\
\label{eq:recurrence_second_solution}
\det({\mathbf{S}}_{k}) = \frac {2{D_{{1}}}^{n}D_{{2}}^{k-n+1} \left[(k-n+1)\gamma_{{1}
}h_{{2}}\theta_{{1}}+{h_{{1}}}n\gamma_{{2}} \right] }{ \left( \gamma_{{1}}\theta_{{1}}h_{{
1}}D_{{2}}+\gamma_{{2}}h_{{2}}D_{{1}} \right)h_1^{2n-1}h_{{2}}^{2(k-n)+1 }}, \quad k = n+2,\hdots,N.
\end{gather}
As all constants appearing in the solutions (\ref{eq:recurrence_first_solution})--(\ref{eq:recurrence_second_solution}) are positive, we have that \change{$\det(\mathbf{S}_k)>0$ for all $k = 1,\hdots,N$. Therefore, all the eigenvalues of $\mathbf{S}$ are positive and hence all the eigenvalues of $\mathbf{A}$ are negative.} 
\\

\noindent\textit{Type GII conditions}\\
\revision{The proof for Type GII conditions (\ref{eq:general_imperfect1})--(\ref{eq:general_imperfect2}) is very similar to the one above and thus omitted.}\\

\noindent The above analysis is valid for the two-layer problem with Dirichlet conditions on both boundaries. Similar analysis can be used to prove that the eigenvalues of $\mathbf{A}$ are negative for problems with more than $m = 2$ layers and other choices of the boundary conditions.
\end{proof}

\change{Using Theorem \ref{thm:eigenvalues} and the fact that $\tau>0$,} clearly $\sigma(\tau\mathbf{A})\subset R_{C}\subset R_{B}$ and hence both the backward Euler (\ref{eq:backward_euler}) and Crank-Nicolson (\ref{eq:crank_nicolson}) schemes are unconditionally stable. \change{We note also that both $\mathbf{I}-\tau\mathbf{A}$ and $\mathbf{I}-\frac{\tau}{2}\mathbf{A}$, appearing in (\ref{eq:backward_euler}) and (\ref{eq:crank_nicolson}) respectively, are invertible as both these matrices have strictly positive eigenvalues.} By bounding the spectrum of $\tau\mathbf{A}$, restrictions on the time step that ensure stability of the forward Euler scheme (\ref{eq:forward_euler}) can be derived. To achieve this we use the Gershgorin circle theorem \citep{thomas_1995}: \change{let} $\mathcal{D}_{p}$ be the Gershgorin disc corresponding to the $p$th row of $\tau\mathbf{A}$:
{
\begin{align}
\label{eq:Gershgorin_disc}
\mathcal{D}_{p} = \left\{z\in\mathbb{C} : |z+c_{p}|\leq r_{p}\right\},\quad
c_{p} = -\tau a_{p,p},\quad r_{p} = \tau(|a_{p,p-1}| + |a_{p,p+1}|).
\end{align}
The} Gershgorin circle theorem states that every eigenvalue of $\tau\mathbf{A}$ lies within at least one of the Gershgorin discs $\mathcal{D}_{p}$, $p = 1,\hdots,N$. {Therefore, by identifying constraints on the time step $\tau$ that ensure that each of the Gershgorin discs lie in the stability region $R_{F}$ (Figure \ref{fig:stability_regions}a), we can derive sufficient conditions for stability \change{of} the forward Euler scheme}. Since $a_{p,p} < 0$ for each of the finite volume equations \change{(\ref{eq:interior_ODEs}), (\ref{eq:left_BC_ODEs})--(\ref{eq:right_BC_dirichlet_ODEs}), (\ref{eq:interface_ODEs3})--(\ref{eq:left_of_interface_ODEs}), (\ref{eq:interface1_ODEs4}) and (\ref{eq:interface2_ODEs4})}, we always have $c_{p}>0$ and hence all the Gershgorin discs are centered along the negative real axis. {Furthermore, as the eigenvalues of $\mathbf{A}$ are real, the Gershgorin discs reduce to intervals along the real axis: $\mathcal{D}_{p} = \left\{x\in\mathbb{R} : |x+c_{p}|\leq r_{p}\right\}$.} {Finally, as the eigenvalues of $\mathbf{A}$ are negative \change{(Theorem \ref{thm:eigenvalues})}, for stability of the forward Euler scheme (\ref{eq:forward_euler}) it is sufficient that:
\begin{align}
\label{eq:forward_condition}
c_{p} + r_{p} \leq 2,
\end{align}
for all $p = 1,\hdots,N$. In the following sections, we derive a \change{set of sufficient stability conditions for the forward Euler scheme} by applying the above constraint to the individual finite volume equations derived in section \ref{sec:spatial_discretisation}.}
\subsubsection{Interior nodes}
Consider the row of the matrix $\tau\mathbf{A}$ corresponding to the finite volume equation (\ref{eq:interior_ODEs}) for nodes located in the interior of the layers. The Gershgorin disc associated with this row takes the form of (\ref{eq:Gershgorin_disc}) with:
\begin{align}
\label{eq:C1}
c_p = \frac{2D_i\tau}{h_i^2}, \quad r_p = \frac{2D_i\tau}{h_i^2}.
\end{align}
\change{In this case, applying the constraint (\ref{eq:forward_condition}) yields the following restriction on the time step}:
\begin{align}
\label{eq:tau_interior}
\tau \leq \frac{h_i^2}{2D_i},
\end{align}
for $i = 1,\hdots,m$, which is precisely the classical stability condition for the single-layer diffusion problem applied to each layer.
\subsubsection{Boundary nodes}
Consider the row of the matrix $\tau\mathbf{A}$ corresponding to the finite volume equation (\ref{eq:left_BC_ODEs}) for the node located at the left boundary of the domain. The Gershgorin disc associated with this row takes the form of (\ref{eq:Gershgorin_disc}) with:
\begin{align}
\label{eq:C2}
c_p = \frac{2D_1\tau}{h_1}\left[\frac{1}{h_1}+\frac{a_L}{b_L}\right], \quad r_p = \frac{2D_1\tau}{h_1^2}.
\end{align}
\change{Applying the constraint (\ref{eq:forward_condition})} yields:
\begin{align}
\label{eq:tau_left}
\tau \leq  \left[\frac{2b_L}{2b_L+a_Lh_1}\right]\frac{h_1^2}{2D_1}.
\end{align}
In the case of a Neumann condition on the left boundary ($a_{L} = 0$ and $b_{L}\neq 0$), the restriction (\ref{eq:tau_left}) is identical to (\ref{eq:tau_interior}) and thus no additional stability restriction is imposed. Recall that for the case of a Dirichlet boundary condition at the left boundary, $b_{L} = 0$, a finite volume equation is not included at the left boundary, which leads to the modified finite volume equation (\ref{eq:left_BC_dirichlet_ODEs}) for the node immediately to the right of the left boundary ($u_{1,1}$). The Gershgorin disc for the row of $\tau\mathbf{A}$ corresponding to this finite volume equation is given by (\ref{eq:Gershgorin_disc}) with:  
\begin{align}
\label{eq:C3}
c_p = \frac{2D_1\tau}{h_1^2}, \quad r_p = \frac{D_1\tau}{h_1^2}.
\end{align}
\change{In this case, applying the constraint} (\ref{eq:forward_condition}) yields:
\begin{align}
\label{eq:left_stability}
\tau \leq \frac{4}{3}\frac{h_1^2}{2D_1}.
\end{align}
\change{Clearly, (\ref{eq:left_stability}) is less restrictive than (\ref{eq:tau_interior})} and hence a Dirichlet or Neumann condition at the left boundary provides no additional stability restriction on the time step.

Similarly, for the right boundary, arising from the finite volume equations \change{(\ref{eq:right_BC_ODEs}) and (\ref{eq:right_BC_dirichlet_ODEs})}, we obtain the stability conditions:
\begin{align}
\label{eq:right_stability}
\tau \leq \frac{4}{3}\frac{h_m^2}{2D_m},
\end{align}
\change{for a Dirichlet boundary condition and 
\begin{align}
\label{eq:tau_right}
\tau \leq  \left[\frac{2b_R}{2b_R+a_Rh_m}\right]\frac{h_m^2}{2D_m},
\end{align}
for a Neumann or Robin boundary condition.}
\change{We remark that the} stability conditions (\ref{eq:tau_left}) and (\ref{eq:tau_right}) are equivalent to those derived by \citet{thomas_1995} for the case of the single-layer diffusion problem.
\subsubsection{Interface nodes}
\label{sec:interface}
We now derive stability conditions for the finite volume equations that arise from spatial discretisation of the interface conditions.
\\

\noindent\textit{Type GI conditions}\\
\change{Consider the Type GI interface conditions (\ref{eq:general_perfect1})--(\ref{eq:general_perfect2}) and the case $\theta_{i}\geq 1$. Recall that for $\theta_{i}\geq 1$ the finite volume equations (\ref{eq:interface_ODEs3}) and (\ref{eq:right_of_interface_ODEs}) are used in the formulation. The Gershgorin disc associated with the row of the matrix $\tau\mathbf{A}$ corresponding to the finite volume equation (\ref{eq:interface_ODEs3}) takes the form  (\ref{eq:Gershgorin_disc}) with:}
\begin{align*}
c_p = \frac{2D_iD_{i+1}\theta_i\tau}{\gamma_ih_i\theta_iD_{i+1}+\gamma_{i+1}h_{i+1}D_i}\left[\frac{\gamma_i}{h_i}+\frac{\gamma_{i+1}}{\theta_ih_{i+1}}\right], \quad r_p =   \frac{2D_iD_{i+1}\theta_i\tau}{\gamma_ih_i\theta_iD_{i+1}+\gamma_{i+1}h_{i+1}D_i}\left[\frac{\gamma_i}{h_i}+\frac{\gamma_{i+1}}{h_{i+1}}\right].
\end{align*}
\change{In this case, the constraint (\ref{eq:forward_condition}) yields the following restriction on the time step:
\begin{gather}
\label{eq:coeff0}
\tau\leq \frac{(\gamma_{i}h_{i}\theta_iD_{i+1} + \gamma_{i+1}h_{i+1}D_{i})h_{i}h_{i+1}}{\left(2\theta_i\gamma_{i}h_{i+1}+(1+\theta_i)\gamma_{i+1}h_{i}\right)D_{i}D_{i+1}}.
\end{gather}
Consider $\theta_{i} = 1$. In this case,} the stability condition (\ref{eq:coeff0}) can be expressed in terms of the classical stability condition \change{in the $i$th layer (\ref{eq:tau_interior})} as follows:
\begin{align}
\label{eq:coeff1}
\tau\leq \left[\frac{\gamma_{i}h_{i}h_{i+1}D_{i+1} + \gamma_{i+1}h_{i+1}^{2}D_{i}}{\gamma_{i}h_{i}h_{i+1}D_{i+1} + \gamma_{i+1}h_{i}^{2}D_{i+1}}\right]\frac{h_{i}^{2}}{2D_{i}},
\end{align}
or equivalently \change{in the ($i+1$)th} layer as:
\begin{align}
\label{eq:coeff2}
\tau\leq \left[\frac{\gamma_{i+1}h_{i}h_{i+1}D_{i} + \gamma_{i}h_{i}^{2}D_{i+1}}{\gamma_{i+1}h_{i}h_{i+1}D_{i} + \gamma_{i}h_{i+1}^{2}D_{i}}\right]\frac{h_{i+1}^{2}}{2D_{i+1}}.
\end{align}
Note that the coefficients of $h_{i}^{2}/(2D_{i})$ and $h_{i+1}^{2}/(2D_{i})$ in equations (\ref{eq:coeff1}) and (\ref{eq:coeff2}) are less than one provided $h_{i+1}^{2}D_{i}<h_{i}^{2}D_{i+1}$ and $h_{i+1}^{2}D_{i}>h_{i}^{2}D_{i+1}$, respectively. \change{Since it is not possible to satisfy these two conditions simultaneously,} it is not possible for  (\ref{eq:coeff0}) to be more restrictive than (\ref{eq:tau_interior}) in both the $i$th and $(i+1)$th layers. \change{For $\theta_{i} > 1$, however, care must be taken when choosing the time step as it is not possible to make such a claim. For example, for the special case $\gamma_{i} = \gamma_{i+1} = D_{i} = D_{i+1}$ and $h_{i} = h_{i+1}$, (\ref{eq:coeff0}) is more restrictive than (\ref{eq:tau_interior}) in both the $i$th and ($i+1$)th layers.}

\change{Next, consider the Gershgorin disc associated with the finite volume equation (\ref{eq:right_of_interface_ODEs}) that governs the temporal behaviour of the solution at the node located immediately to the right of the $i$th interface ($u_{i+1,1}$). This Gershgorin disc takes the form of (\ref{eq:Gershgorin_disc}) with:}
\begin{align*}
c_p = \frac{2D_{i+1}\tau}{h_{i+1}^2}, \quad r_p = \frac{D_{i+1}\tau}{h_{i+1}^2}\left[\frac{1}{\theta_i}+1\right].
\end{align*}
\change{In this case,} the constraint (\ref{eq:forward_condition}) yields:
\begin{align}
\label{eq:coeff2.3}
\tau \leq \left[\frac{4\theta_i}{3\theta_i+1}\right]\frac{h_{i+1}^2}{2D_{i+1}} .
\end{align}
\change{Clearly, if $\theta_i > 1$, (\ref{eq:coeff2.3}) is less restrictive than  (\ref{eq:tau_interior}) and equally restrictive if $\theta_i = 1$. Therefore, no additional stability restriction is imposed.} 

\change{We now consider Type GI interface conditions (\ref{eq:general_perfect1})--(\ref{eq:general_perfect2}) with $\theta_{i}< 1$, where the finite volume equations (\ref{eq:reformulation}) and (\ref{eq:left_of_interface_ODEs}) are instead utilised in the formulation.
The Gershgorin disc of $\tau\mathbf{A}$ associated with the finite volume equation (\ref{eq:reformulation}) takes the form (\ref{eq:Gershgorin_disc}) with:}
\begin{align*}
c_p = \frac{2D_iD_{i+1}\tau}{\gamma_ih_i\theta_iD_{i+1}+\gamma_{i+1}h_{i+1}D_i}\left[\frac{\theta_i\gamma_i}{h_i}+\frac{\gamma_{i+1}}{h_{i+1}}\right], \quad r_p =   \frac{2D_iD_{i+1}\tau}{\gamma_ih_i\theta_iD_{i+1}+\gamma_{i+1}h_{i+1}D_i}\left[\frac{\gamma_i}{h_i}+\frac{\gamma_{i+1}}{h_{i+1}}\right],
\end{align*}
\change{and applying the constraint (\ref{eq:forward_condition}) yields:
\begin{gather}
\label{eq:coeff2.4}
\tau\leq \frac{(\gamma_{i}h_{i}\theta_iD_{i+1} + \gamma_{i+1}h_{i+1}D_{i})h_{i}h_{i+1}}{\left((1+\theta_i)\gamma_{i}h_{i+1}+2\gamma_{i+1}h_{i}\right)D_{i}D_{i+1}}.
\end{gather}
As was the case for the stability condition (\ref{eq:coeff0}), (\ref{eq:coeff2.4}) may or may not be more restrictive than (\ref{eq:tau_interior}) so care must be taken when choosing the time step. }

\change{The Gershgorin disc associated with the row of $\tau\mathbf{A}$ corresponding to the finite volume equation (\ref{eq:left_of_interface_ODEs}) takes the form of (\ref{eq:Gershgorin_disc}) with:}
\begin{align}
\label{eq:C5.2}
c_p = \frac{2D_{i}\tau}{h_{i}^2}, \quad r_p = \frac{(1+\theta_i)D_{i}\tau}{h_{i}^2}.
\end{align}
\change{Applying the constraint (\ref{eq:forward_condition}) yields:}
\begin{align}
\label{eq:coeff2.7}
\tau \leq \left[\frac{4}{3\theta_i+1}\right]\frac{h_{i}^2}{2D_{i}} .
\end{align}
As this formulation is used for the case of $\theta_i < 1$, the condition (\ref{eq:coeff2.7}) is no more restrictive than (\ref{eq:tau_interior}).

\change{In summary, for $\theta_{i} = 1$, the Type GI conditions (\ref{eq:general_perfect1})--(\ref{eq:general_perfect2}) give rise to stability conditions for the forward Euler scheme (\ref{eq:forward_euler}) that are less restrictive than the classical stability condition (\ref{eq:tau_interior}) in the $i$th and ($i+1$)th layers. However, care must be taken for $\theta_{i}\neq 1$, as the stability conditions (\ref{eq:coeff0}) ($\theta_{i} > 1$) and (\ref{eq:coeff2.4}) ($\theta_{i}<1$) may or may not be more restrictive.}\\

\noindent\textit{Type GII conditions}\\
We now study stability restrictions arising from the Type GII interface conditions \change{(\ref{eq:general_imperfect1})--(\ref{eq:general_imperfect2})}. Applying the Gershgorin circle theorem to the row of the matrix $\tau\mathbf{A}$ corresponding to the finite volume equation (\ref{eq:interface1_ODEs4}) yields a Gershgorin disc of the form (\ref{eq:Gershgorin_disc}) with:
\begin{align}
\label{eq:C6}
c_p  = \frac{2D_i\tau}{\gamma_ih_i}\left[H_i+\frac{\gamma_i}{h_i}\right], \quad r_p =  \frac{2D_i\tau}{\gamma_ih_i}\left[\theta_iH_i+\frac{\gamma_i}{h_i}\right].
\end{align}
Similarly, for the finite volume equation (\ref{eq:interface2_ODEs4})  we have:
\begin{align}
\label{eq:C7}
c_p = \frac{2D_{i+1}\tau}{\gamma_{i+1}h_{i+1}}\left[\theta_iH_i+\frac{\gamma_{i+1}}{h_{i+1}}\right],\quad r_p = \frac{2D_{i+1}\tau}{\gamma_{i+1}h_{i+1}}\left[H_i+\frac{\gamma_{i+1}}{h_{i+1}}\right].
\end{align}
Applying the constraint (\ref{eq:forward_condition}) to (\ref{eq:C6}) and (\ref{eq:C7}) yields:
\begin{gather}
\label{eq:coeff3}
\tau \leq \left[\frac{2\gamma_i}{(1+\theta_i)H_ih_{i}+2\gamma_{i}}\right]\frac{h_i^2}{2D_i},\\
\label{eq:coeff4}
\tau \leq \left[\frac{2\gamma_{i+1}}{(1+\theta_i)H_ih_{i+1}+2\gamma_{i+1}}\right]\frac{h_{i+1}^2}{2D_{i+1}},
\end{gather}
\change{respectively.} The coefficients of $h_{i}^{2}/(2D_{i})$ and $h_{i+1}^{2}/(2D_{i})$ in equations \change{(\ref{eq:coeff3}) and (\ref{eq:coeff4})} are less than one if \change{$(1+\theta_{i})H_{i}h_{i} > 0$ and $(1+\theta_{i})H_{i}h_{i+1} > 0$, which is always true,} and hence the stability conditions (\ref{eq:coeff3}) and (\ref{eq:coeff4}) are \change{always} more restrictive than the classical stability condition (\ref{eq:tau_interior}) in the $i$th and $(i+1)$th layers, respectively. Hence, under the forward Euler scheme (\ref{eq:forward_euler}), the finite volume equations (\ref{eq:interface1_ODEs4}) and (\ref{eq:interface2_ODEs4}) impose the following stability restriction on the time step:
\change{\begin{align}
\tau\leq\min\left\{\left[\frac{2\gamma_i}{(1+\theta_i)H_ih_{i}+2\gamma_{i}}\right]\frac{h_i^2}{2D_i},\left[\frac{2\gamma_{i+1}}{(1+\theta_i)H_ih_{i+1}+2\gamma_{i+1}}\right]\frac{h_{i+1}^2}{2D_{i+1}}\right\}.
\end{align}
Table \ref{tab:stability} summarises each of the stability conditions derived in this section for the forward Euler scheme (\ref{eq:forward_euler}).}

\begin{table}[!t]
\def\arraystretch{1.5}
\begin{center}
\begin{tabular}{ |l|c|c| } 
 \hline
    & \textbf{Finite volume} & \\[-0.3cm] 
  \textbf{Type of node} & \textbf{equation} & \textbf{Stability condition} \\ 
 \hline
\textbf{Interior} & (\ref{eq:interior_ODEs}) & $\tau\leq\frac{h_i^2}{2D_i}$ \\ 
  \hline
  \textbf{Left boundary} & &\\
   Dirichlet &   (\ref{eq:left_BC_dirichlet_ODEs})& No additional restriction \\ 
 Neumann & (\ref{eq:left_BC_ODEs}) &  No additional restriction
 \\ 
  Robin ($b_L \neq 0$) &(\ref{eq:left_BC_ODEs}) &  $\tau \leq  \left[\frac{2b_L}{2b_L+a_Lh_1}\right]\frac{h_1^2}{2D_1}$
 \\ 
  \hline
 \textbf{Right boundary} & &\\
 Dirichlet &  (\ref{eq:right_BC_dirichlet_ODEs}) & No additional restriction \\ 
 Neumann & (\ref{eq:right_BC_ODEs}) &  No additional restriction \\ 
 Robin ($b_R \neq 0$) & (\ref{eq:right_BC_ODEs})  &  $\tau \leq  \left[\frac{2b_R}{2b_R+a_Rh_m}\right] \frac{h_m^2}{2D_m} \quad $ \\ 
 \hline
 \textbf{Interface} & &\\
 Type GI \change{($\theta_{i} \geq 1$)} & (\ref{eq:interface_ODEs3}) &  No additional restriction ($\theta_{i} = 1$) \\
 & & $\tau\leq \frac{(\gamma_{i}h_{i}\theta_iD_{i+1} + \gamma_{i+1}h_{i+1}D_{i})h_{i}h_{i+1}}{\left(2\theta_i\gamma_{i}h_{i+1}+(1+\theta_i)\gamma_{i+1}h_{i}\right)D_{i}D_{i+1}}$ ($\theta_{i}\neq 1$)\\ 
   & (\ref{eq:right_of_interface_ODEs}) &  No additional restriction \\ 
 Type GI \change{($\theta_{i} < 1$)}   & (\ref{eq:reformulation}) &  $\tau\leq \frac{(\gamma_{i}h_{i}\theta_iD_{i+1} + \gamma_{i+1}h_{i+1}D_{i})h_{i}h_{i+1}}{\left((1+\theta_i)\gamma_{i}h_{i+1}+2\gamma_{i+1}h_{i}\right)D_{i}D_{i+1}}$ \\ 
    & (\ref{eq:left_of_interface_ODEs}) &  No additional restriction \\ 
Type GII & (\ref{eq:interface1_ODEs4})--(\ref{eq:interface2_ODEs4}) & $\tau\leq\min\left\{\left[\frac{2\gamma_i}{(1+\theta_i)H_ih_{i}+2\gamma_{i}}\right]\frac{h_i^2}{2D_i},\left[\frac{2\gamma_{i+1}}{(1+\theta_i)H_ih_{i+1}+2\gamma_{i+1}}\right]\frac{h_{i+1}^2}{2D_{i+1}}\right\}$\\
 \hline
\end{tabular}
\caption{Stability conditions for the forward Euler scheme (\ref{eq:forward_euler}) arising from the different \change{finite volume equations. Note that both the backward Euler (\ref{eq:backward_euler}) and Crank-Nicolson (\ref{eq:crank_nicolson}) schemes are not included as both are unconditionally stable}. {No additional restriction means that the stability condition arising from the finite volume equation is not more restrictive than the interior stability condition \change{(\ref{eq:tau_interior}) in both the $i$th and ($i+1$)th layers.}}}
\label{tab:stability}
\end{center}
\end{table}
\subsection{Convergence}
We demonstrate convergence of our finite volume method via the Lax equivalence theorem \citep{lax_1956}, which states that a consistent finite difference/volume method for a well-posed linear initial value problem is convergent if and only if it is stable \citep{strikwerda_1989}. Note that the linear initial value problem (\ref{eq:PDE})--(\ref{eq:right_BC}) is well-posed as it is known to have an exact solution \change{(see, e.g., \citep[Appendix C]{carr_2016c})}. Hence, in order to prove that our finite volume scheme is convergent, we prove that it is consistent under the assumption that the stability conditions, derived in Section \ref{sec:stability} and summarised in Table \ref{tab:stability}, hold.

As an example, let us consider the forward Euler discretisation of the finite volume equation (\ref{eq:interface_ODEs3}):
\begin{align}
\label{eq:forward_GI1}
u_{i,n}^{(k+1)} = u_{i,n}^{(k)} +\frac{2D_iD_{i+1}\theta_i\tau}{\gamma_ih_i\theta_iD_{i+1}+\gamma_{i+1}h_{i+1}D_i}\left[\frac{\gamma_i}{h_i}u_{i,n-1}^{(k)} - \left(\frac{\gamma_i}{h_i}+\frac{\gamma_{i+1}}{\theta_ih_{i+1}}\right)u_{i,n}^{(k)}+\frac{\gamma_{i+1}}{h_{i+1}}u_{i+1,1}^{(k)}\right],
\end{align}
\change{where the notation $u_{i,j}^{(k)}$ is used to denote the discrete numerical approximation to $u_{i}(x_{i,n},t_{k})$, the exact solution in the $i$th layer, $u_{i}(x,t)$, evaluated at $x = x_{i,j}$ and $t = t_{k}$.} For consistency, we require that the local truncation error tends to zero as the time step and grid spacing tend to zero. The local truncation error (LTE) corresponding to (\ref{eq:forward_GI1}) is defined as:
\begin{multline}
\label{eq:LTE1}
\text{LTE} = \frac{u_i(x_{i,n},t_{k+1}) - u_{i}(x_{i,n},t_k)}{\tau}-\frac{2D_iD_{i+1}\theta_i}{\gamma_ih_iD_{i+1}\theta_i+\gamma_{i+1}h_{i+1}D_i}\Bigg[\frac{\gamma_i}{h_i}u_{i}(x_{i,n-1},t_k) \\ -\left(\frac{\gamma_i}{h_i}+\frac{\gamma_{i+1}}{\theta_ih_{i+1}}\right)u_{i}(x_{i,n},t_k)+\frac{\gamma_{i+1}}{h_{i+1}}u_{i+1}(x_{i+1,1},t_k)\Bigg].
\end{multline}
Expanding the exact solution in equation (\ref{eq:LTE1}) using a Taylor series about $t = t_{k}$ and $x = x_{i,n}$ and simplifying the result produces:
\begin{gather}
\label{eq:LTE3}
\text{LTE} = \frac{\partial u_{i}}{\partial t}(x_{i,n},t_k)+O(\tau)-\frac{2D_iD_{i+1}\theta_i}{\gamma_ih_i\theta_iD_{i+1}+\gamma_{i+1}h_{i+1}D_i}\Bigg[\frac{\gamma_{i+1}}{h_{i+1}}\Bigg(u_{i+1}(x_{i,n},t_{k}) - \frac{u_{i}(x_{i,n},t_{k})}{\theta_i}\Bigg)\nonumber\\ + \gamma_{i+1}\frac{\partial u_{i+1}}{\partial x}(x_{i,n},t_{k}) - \gamma_{i}\frac{\partial u_{i}}{\partial x}(x_{i,n},t_{k})\nonumber\\ + \frac{\gamma_ih_i}{2}\frac{\partial^2 u_{i}}{\partial x^2}(x_{i,n},t_k) +\frac{\gamma_{i+1}h_{i+1}}{2}\frac{\partial^2 u_{i+1}}{\partial x^2}(x_{i,n},t_k) + O(h_i^2) + O(h_{i+1}^2)\Bigg].
\end{gather}
Applying the interface conditions (\ref{eq:general_perfect1}) and (\ref{eq:general_perfect2}) and noting the equality implied by the diffusion equation (\ref{eq:PDE}) yields:
\begin{multline}
\label{eq:LTE5}
\text{LTE} = \frac{\partial u_{i}}{\partial t}(x_{i,n},t_k)+O(\tau)-\frac{D_iD_{i+1}\theta_i}{\gamma_ih_i\theta_iD_{i+1}+\gamma_{i+1}h_{i+1}D_i}\Bigg[\frac{\gamma_ih_i}{D_i}\frac{\partial u_{i}}{\partial t}(x_{i,n},t_k) + O(h_i^2) \\+\frac{\gamma_{i+1}h_{i+1}}{D_{i+1}}\frac{\partial u_{i+1}}{\partial t}(x_{i,n},t_k) + O(h_{i+1}^2)\Bigg].
\end{multline}
As $u_{i}(x_{i,n},t) =\theta_i u_{i+1}(x_{i,n},t)$, \change{it follows that} the temporal derivatives of each function also follow a similar relation: $\partial u_{i}/\partial t (x_{i,n},t) = \theta_i\partial u_{i+1}/\partial t (x_{i,n},t)$. Substituting this latter result into equation (\ref{eq:LTE5}) and simplifying gives:
\begin{multline}
\label{eq:LTE6}
\text{LTE} = \frac{\partial u_{i}}{\partial t}(x_{i,n},t_k)+O(\tau)-\frac{D_iD_{i+1}\theta_i}{\gamma_ih_i\theta_iD_{i+1}+\gamma_{i+1}h_{i+1}D_i}\Bigg[\frac{\gamma_ih_i}{D_i}\frac{\partial u_{i}}{\partial t}(x_{i,n},t_k) + O(h_i^2) \\+\frac{\gamma_{i+1}h_{i+1}}{\theta_iD_{i+1}}\frac{\partial u_{i}}{\partial t}(x_{i,n},t_k) + O(h_{i+1}^2)\Bigg].
\end{multline}
\change{Combining the temporal derivative terms, (\ref{eq:LTE6}) finally} simplifies to $\change{\text{LTE}} = O(\tau + h_{i} + h_{i+1})$.

Since $\text{LTE} \rightarrow 0$ as $\tau \rightarrow 0$, $h_i \rightarrow 0$ and $h_{i+1} \rightarrow 0$}, the finite volume discretisation (\ref{eq:forward_GI1}) is consistent. Applying similar analysis to the other types of finite volume equations \change{(\ref{eq:interior_ODEs}), (\ref{eq:left_BC_ODEs})--(\ref{eq:right_BC_dirichlet_ODEs}), (\ref{eq:interface_ODEs3})--(\ref{eq:left_of_interface_ODEs}), (\ref{eq:interface1_ODEs4}) and (\ref{eq:interface2_ODEs4})}  leads to the same conclusion for all three time discretisation methods. Therefore, \change{we conclude that the finite volume method developed in Section \ref{sec:discretisation}} is convergent provided the time step is chosen to ensure stability.  \section{Numerical results}
\label{sec:numerical_experiments}
\subsection{Error analysis}
\label{sec:error_analysis}
To assess the accuracy of our new finite volume method, we consider four test cases considered previously by \citet{carr_2016c}, which we will refer to as Cases \change{A--D}. These cases involve $m = 2$ layers, domain $[l_0,l_1,l_2] = [0,0.5,1]$, diffusivities $D_1 = 1$ and $D_2 = 0.1$, and external boundary data $a_L = 1$, $b_L = 0$, $c_L = 1$, $a_R = 0$, $b_R = 1$, $c_R = 0$. The test cases assess each of the \change{four} types of interface conditions:
\begin{itemize}
\item Case A: Type I conditions (\ref{eq:type1interface1})--(\ref{eq:type1interface2}).
\item Case B: Type II conditions (\ref{eq:type2interface1})--(\ref{eq:type2interface2}) and $H_1 = 0.5$.
\item Case C: Type IV conditions (\ref{eq:type4interface1})--(\ref{eq:type4interface2}) with $\theta_1 = 1.2$.
\item Case D: Type III conditions (\ref{eq:type3interface1})--(\ref{eq:type3interface2}) with $\gamma_1 = \gamma_2 = 2$.
\end{itemize}
 
\begin{figure}[!t]
\centering
\subfloat[Case A]{\includegraphics[width=0.5\textwidth]{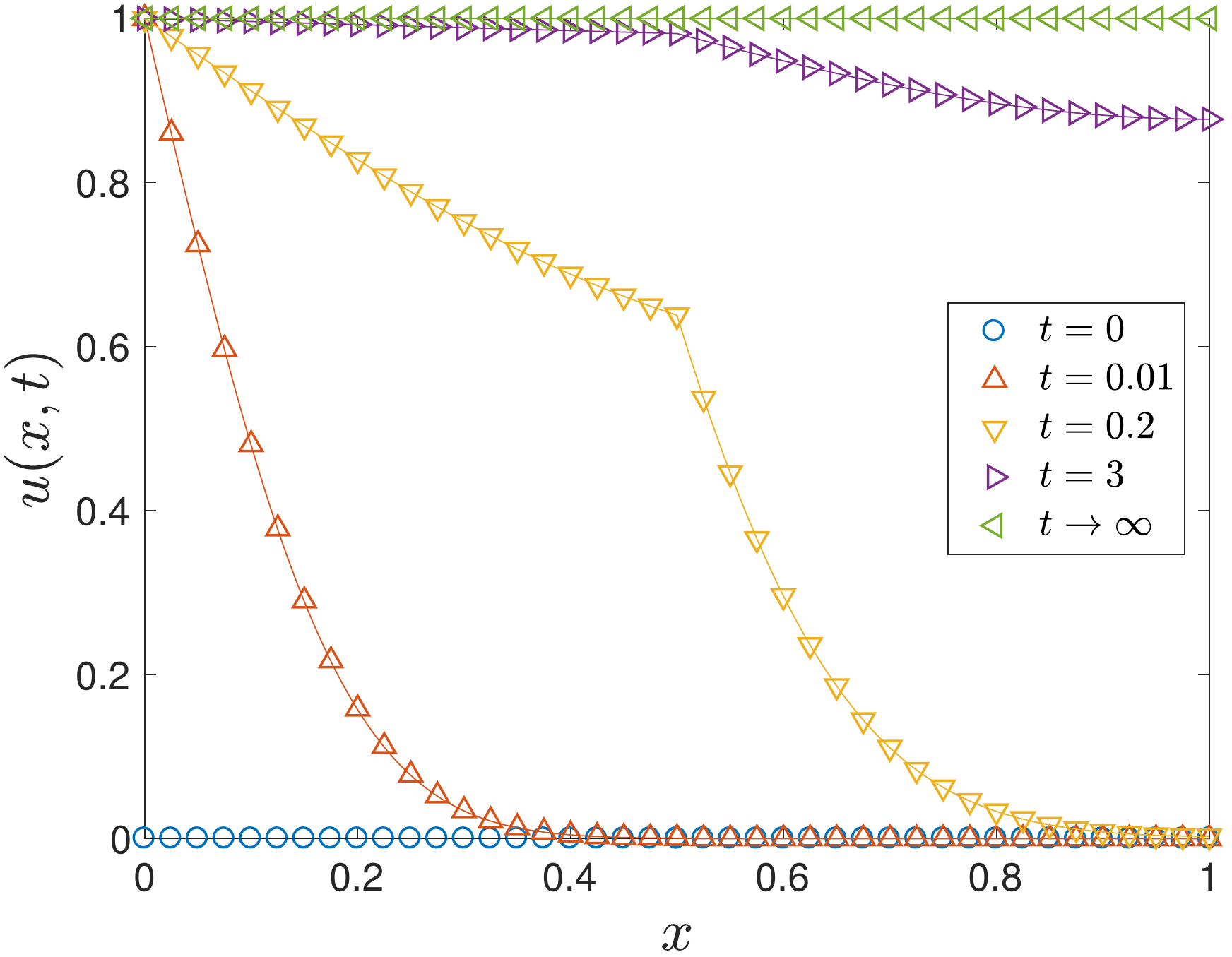}}
\subfloat[Case B]{\includegraphics[width=0.5\textwidth]{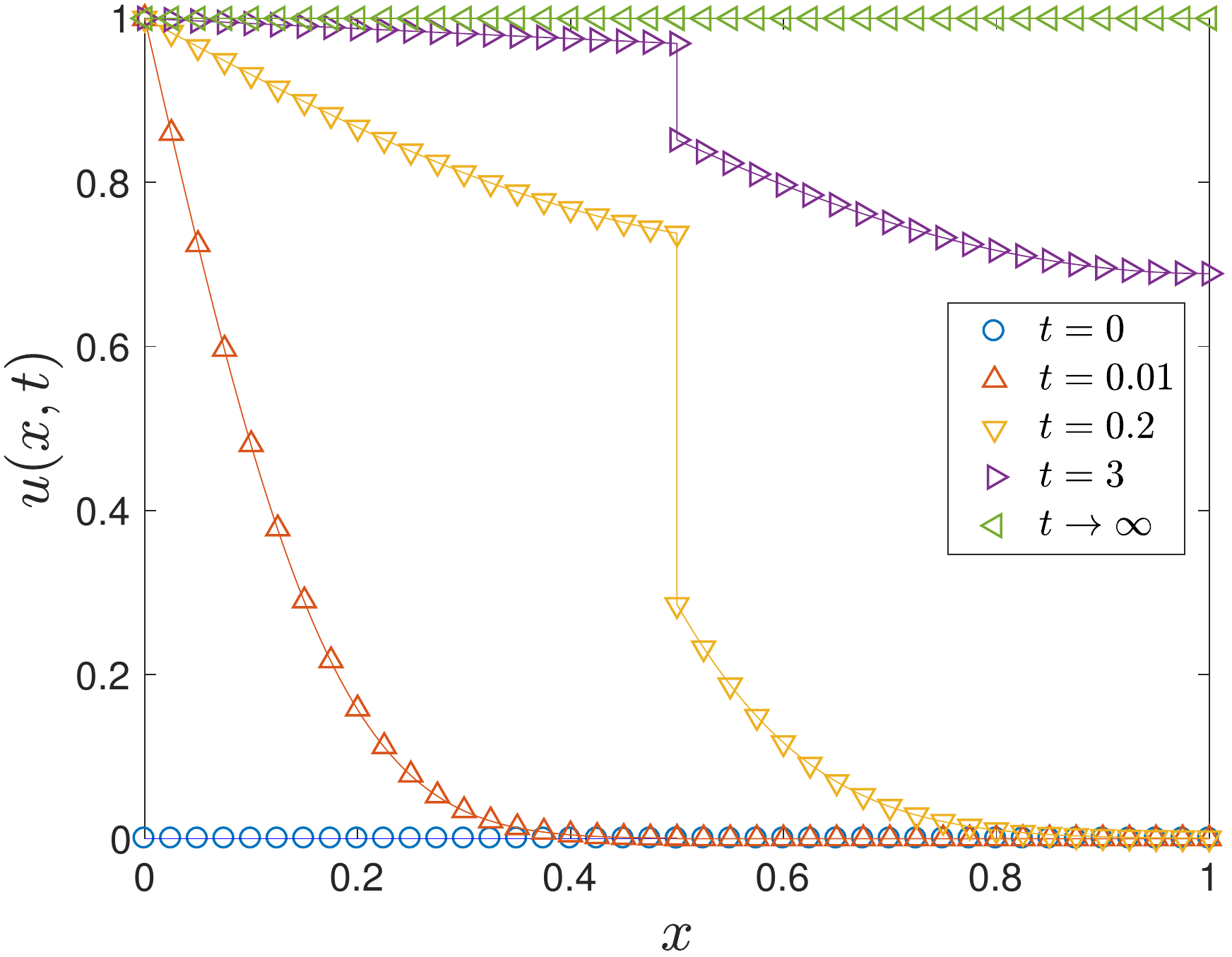}}\\
\subfloat[Case C]{\includegraphics[width=0.5\textwidth]{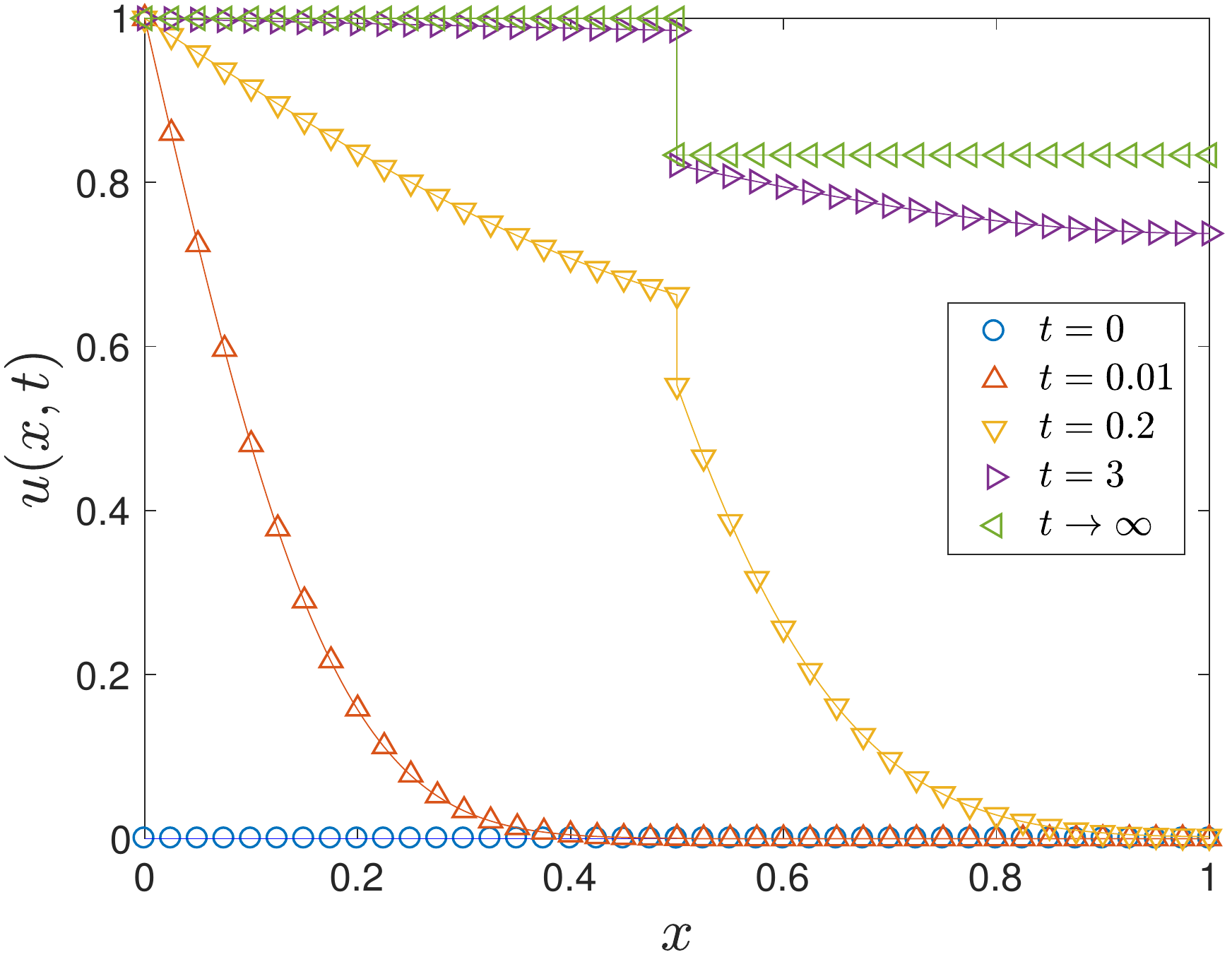}}
\subfloat[Case D]{\includegraphics[width=0.5\textwidth]{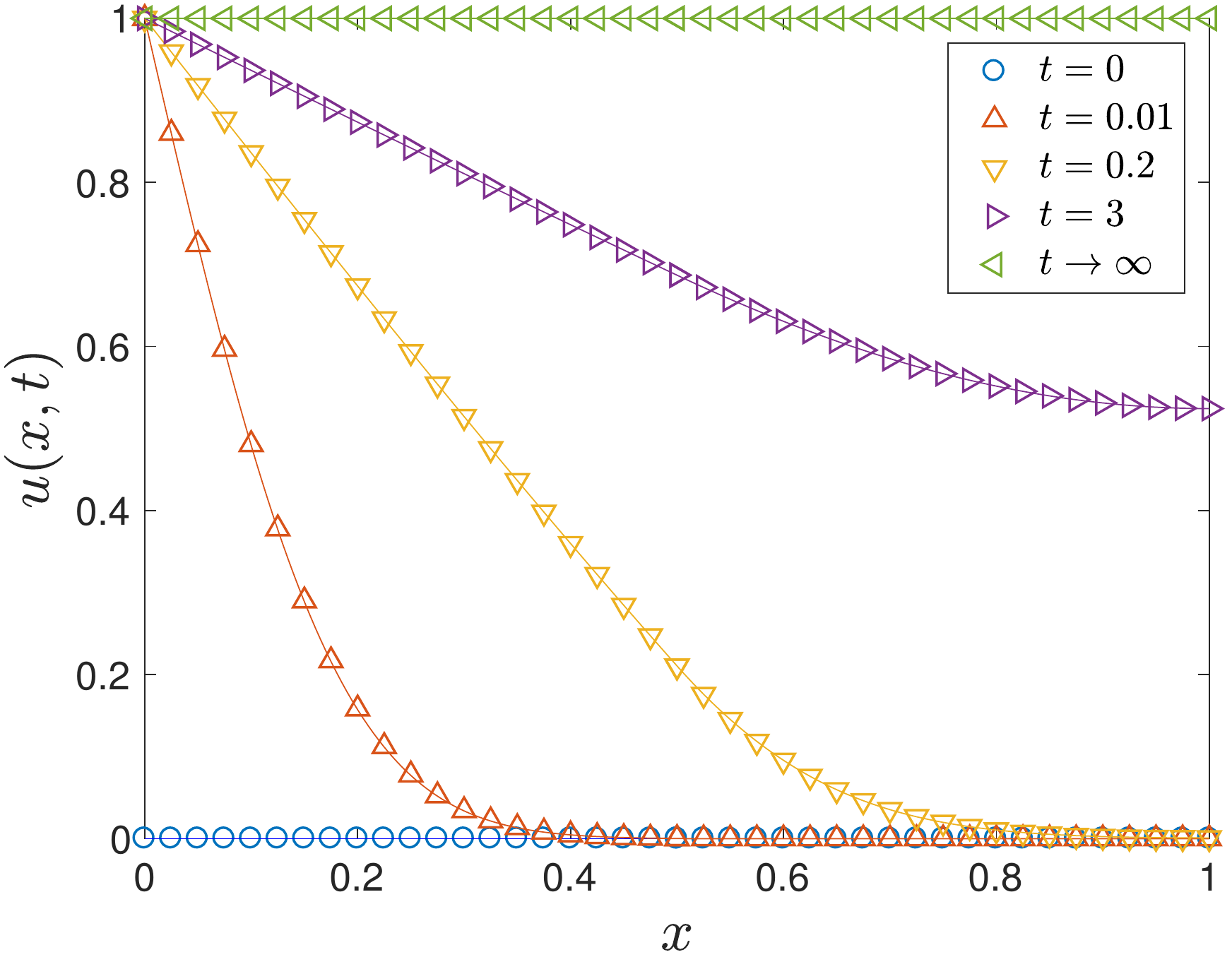}}
\caption{Solution verification of the new finite volume method for Cases A--D. Dots indicate the solution obtained using the  new finite volume scheme with a forward Euler temporal discretisation, node spacing $h_{i} = 0.025$ in both layers and a time step of $\tau = 10^{-4}$, while the continuous lines represent the analytical solution.}
\label{fig:CaseABCD}
\end{figure}

\revision{These four test cases provide a good test for our numerical schemes as they assess the different types of behaviour found at the interfaces in multilayer diffusion problems. This is evident in Figure \ref{fig:CaseABCD}, which provides the solution of each test case at selected points in time. Case A (Figure \ref{fig:CaseABCD}a) is a standard layered diffusion problem: the gradient of the solution is discontinuous at the interface as the conductivities in the layers, $\gamma_1$ and $\gamma_2$, are not equal. Case B (Figure \ref{fig:CaseABCD}b) mimics the physical situation of a thin resistive layer at the interface: the finite contact transfer coefficient, $H_1$, produces a discontinuity in the solution between the two layers with $u_{1}(l_{1},t)\neq u_{2}(l_{1},t)$. For this interface condition, the jump discontinuity $u_{1}(l_{1},t) - u_{2}(l_{1},t)$ is directly proportional to the gradient of $u_{1}(x,t)$ (or $u_{2}(x,t)$) at the interface \cite{carr_2016c} and therefore, as can be seen, decreases in magnitude as time progresses. For Case C (Figure \ref{fig:CaseABCD}c), the solution is discontinuous as the partition coefficient $\theta_1\neq 1$. For this problem, the jump discontinuity is directly proportional to the value of $u_{2}(l_{1},t)$ (or $u_{1}(l_{1},t)$) \cite{carr_2016c} and therefore increases in magnitude for increasing time. For Case D (Figure \ref{fig:CaseABCD}d), both the solution and gradient are continuous as $H_1 \rightarrow \infty$, $\theta_1= 1$ and $\gamma_1 = \gamma_2$.}

Additionally, Figure \ref{fig:CaseABCD} compares the numerical solution obtained using the new finite volume method to an analytical solution. The analytical solution is the \change{classical eigenfunction expansion solution (see, e.g., \cite[Appendix C]{carr_2016c})} truncated to include only the first 100 terms in the summation. In all cases, the finite volume solution is in excellent agreement with the analytical solution, which confirms the correctness of our treatment of the interface conditions. We compute the relative error as:
\begin{align}
\label{eq:relative_error}
{\text{Error}(t_{k})} = \frac{\max\limits_{i,j}\,\bigl|u_{i}(x_{i,j},t_{k}) -  u_{i,j}^{(k)}\bigr|}{\max\limits_{i,j}\,\bigl|u_i(x_{i,j},t_k)\bigr|},
\end{align}
\change{where we recall that $u_{i,j}^{(k)}$ is the discrete numerical approximation to the exact solution $u_{i}(x_{i,j},t_{k})$ and the maximum is taken over all layers ($i = 1,\hdots,m$) and all nodes ($j = 0,\hdots,n$).}

To determine the order of spatial accuracy for the finite volume schemes we investigate the reduction in the relative error (\ref{eq:relative_error}) as the grid spacing is reduced. This is carried out using a fixed time step $\tau$ and a uniform grid spacing across the entire domain: $h_{i} = h$ for all $i = 1,\hdots,m$. We solve Cases A--D using $\tau = 10^{-7}$ and $h = 2^{-3},2^{-4},\hdots,2^{-7}$ and compute the relative error {(\ref{eq:relative_error})} at $t_{k} = 0.2$ ($k = 2 \times 10^6$ time steps). The value of the time step is chosen to ensure that the spatial error dominates over the temporal error as we expect an error of $O(\tau + h^{2})$ for the forward Euler (\ref{eq:forward_euler}) and backward Euler (\ref{eq:backward_euler}) schemes and $O(\tau^{2}+h^{2})$ for the Crank-Nicolson (\ref{eq:crank_nicolson}) scheme.

In Table \ref{tab:FE_error}, the relative errors for each test case are tabulated for the chosen grid spacings and a forward Euler time discretisation. Similar results are given in Tables \ref{tab:BE_error} and \ref{tab:CN_error} for backward Euler and Crank-Nicolson. In each of these tables, we also compute the ratio of successive errors. For example, in Table \ref{tab:FE_error}, the value of 4.10 for Case A is calculated as 8.01e-03/1.95e-03. These ratios demonstrate that reducing the time step by a factor of two leads to a reduction in the relative error by approximately a factor of four for all three time discretisation methods. Therefore, we conclude that the new finite volume method is second-order accurate in space.

\begin{table}[H]
\centering
\begin{tabular}{|c|c|c|c|c|c|c|c|c|}
\hline
\textbf{Forward Euler} &  \multicolumn{2}{c|}{Case A} & \multicolumn{2}{c|}{Case B} & \multicolumn{2}{c|}{Case C} & \multicolumn{2}{c|}{Case D}\\
\hline
Node spacing & Error & Ratio & Error & Ratio & Error & Ratio & Error & Ratio \\
\hline
\rule{0pt}{2.5ex}$2^{-3}$ &    8.01e-03  & -        & 8.99e-03   & -      &  7.11e-03 & -	  & 1.13e-02  & - \\
$2^{-4}$ &    1.95e-03  & 4.10  & 1.94e-03   & 4.64 & 1.73e-03 & 4.10 & 2.50e-03  & 4.52 \\
$2^{-5}$ &    4.92e-04  & 3.97  & 4.63e-04   & 4.19 & 4.36e-04 & 3.97 & 6.06e-04  & 4.12 \\
$2^{-6}$ &    1.24e-04  & 3.98  & 1.13e-04   & 4.08 & 1.10e-04 & 3.98 & 1.50e-04  & 4.03 \\
$2^{-7}$ &    3.10e-05  & 3.99  & 2.80e-05   & 4.05 & 2.75e-05 & 3.99 & 3.75e-05  & 4.01 \\
\hline
\end{tabular}
\caption{\change{Relative errors and ratios of errors for the forward Euler scheme (\ref{eq:forward_euler}) using a time step of $\tau = 10^{-7}$ and different node spacing.}} 
\label{tab:FE_error}
\end{table}

\begin{table}[H]
\centering
\begin{tabular}{|c|c|c|c|c|c|c|c|c|}
\hline
\textbf{Backward Euler} &  \multicolumn{2}{c|}{Case A} & \multicolumn{2}{c|}{Case B} & \multicolumn{2}{c|}{Case C} & \multicolumn{2}{c|}{Case D}\\
\hline
Node spacing & Error & Ratio & Error & Ratio & Error & Ratio & Error & Ratio\\
\hline
\rule{0pt}{2.5ex}$2^{-3}$ &    8.01e-03  & -        & 8.99e-03   & -      &  7.11e-03 & -	  & 1.13e-02  & - \\
$2^{-4}$ &    1.95e-03  & 4.10  & 1.94e-03   & 4.64 & 1.73e-03 & 4.10 & 2.50e-03  & 4.52 \\
$2^{-5}$ &    4.92e-04  & 3.97  & 4.63e-04   & 4.19 & 4.36e-04 & 3.97 & 6.07e-04  & 4.11 \\
$2^{-6}$ &    1.24e-04  & 3.97  & 1.14e-04   & 4.08 & 1.10e-04 & 3.97 & 1.51e-04  & 4.03 \\
$2^{-7}$ &    3.12e-05  & 3.98  & 2.82e-05   & 4.03 & 2.76e-05 & 3.97 & 3.76e-05  & 4.01 \\
\hline
\end{tabular}
\caption{\change{Relative errors and ratios of errors for the backward Euler scheme (\ref{eq:backward_euler}) using a time step of $\tau = 10^{-7}$ and different node spacing.}}
\label{tab:BE_error}
\end{table}

\begin{table}[H]
\centering
\begin{tabular}{|c|c|c|c|c|c|c|c|c|}
\hline
\textbf{Crank-Nicolson} &  \multicolumn{2}{c|}{Case A} & \multicolumn{2}{c|}{Case B} & \multicolumn{2}{c|}{Case C} & \multicolumn{2}{c|}{Case D}\\
\hline
Node spacing & Error & Ratio & Error & Ratio & Error & Ratio & Error & Ratio\\
\hline
\rule{0pt}{2.5ex}$2^{-3}$ &    8.01e-03  & -        & 8.99e-03   & -      &  7.11e-03 & -	  & 1.13e-02  & - \\
$2^{-4}$ &    1.95e-03  & 4.10  & 1.94e-03   & 4.64 & 1.73e-03 & 4.10 & 2.50e-03  & 4.52 \\
$2^{-5}$ &    4.92e-04  & 3.97  & 4.63e-04   & 4.19 & 4.36e-04 & 3.97 & 6.07e-04  & 4.12 \\
$2^{-6}$ &    1.24e-04  & 3.97  & 1.14e-04   & 4.08 & 1.10e-04 & 3.97 & 1.51e-04  & 4.03 \\
$2^{-7}$ &    3.11e-05  & 3.98  & 2.81e-05   & 4.04 & 2.76e-05 & 3.98 & 3.75e-05  & 4.01 \\
\hline
\end{tabular}
\caption{\change{Relative errors and ratios of errors for the Crank-Nicolson scheme (\ref{eq:crank_nicolson}) using a time step of $\tau = 10^{-7}$ and different node spacing.}}
\label{tab:CN_error}
\end{table}

\subsection{Stability conditions}
In this section, we demonstrate the importance of accounting for the interface conditions when studying the stability restrictions placed on the time step. In particular, for the forward Euler scheme (\ref{eq:forward_euler}), we highlight the risk of choosing a time step that satisfies only the classical stability condition $\tau\leq h_{i}^{2}/(2D_{i})$ in each layer.

As a first example, consider Case B, as described in Section \ref{sec:numerical_experiments}, with a contact transfer coefficient of $H_{1} = 5$ and grid spacing $h_{1} = h_{2} = 0.025$. To ensure stability of this problem (which we \change{label} as Case E) under the forward Euler scheme (\ref{eq:forward_euler}), {it is sufficient for the time step to satisfy} the following constraints from Table \ref{tab:stability}:
\begin{alignat}{2}
\label{eq:caseE_timestep1}
&\text{Layer 1:} &\quad \tau &\leq \frac{h_{1}^{2}}{2D_{1}} = 3.125 \times 10^{-4},\\
\label{eq:caseE_timestep2}
&\text{Layer 2:} &\quad \tau &\leq \frac{h_{2}^{2}}{2D_{2}} = 3.125 \times 10^{-3},\\
\label{eq:caseE_timestep3}
&\text{Interface:}&\quad \tau & \leq \min\left\{\frac{2\gamma_1h_1^2}{2D_1\left[(1+\theta_1)H_1h_{1}+2\gamma_{1}\right]}, \frac{2\gamma_{2}h_{2}^2}{2D_{2}\left[(1+\theta_1)H_1h_{2}+2\gamma_{2}\right]}\right\} \change{=2.78\times 10^{-4}}.
\end{alignat}
Naively considering only the classical stability conditions (\ref{eq:caseE_timestep1}) and (\ref{eq:caseE_timestep2}) and choosing a time step of $3.125\times 10^{-4}$ leads to \change{an} unstable solution as $\rho(\mathbf{A}_{F}) = 1.00873$. However, correctly taking into account the \change{stability condition arising at the interface (\ref{eq:caseE_timestep3})} and choosing a time step of $\tau = 2.78\times 10^{-4}$ yields a spectral radius of $\rho(\mathbf{A}_{F}) = 0.99979$ and a stable solution as shown in Figure \ref{fig:Stability}a.

As already mentioned in Section \ref{sec:interface} the stability conditions (\ref{eq:coeff3}) and (\ref{eq:coeff4}) are always more restrictive than the classical stability conditions $\tau\leq h_{i}^{2}/(2D_{i})$ and $\tau\leq h_{i+1}^{2}/(2D_{i+1})$, respectively. Moreover, if \change{$(1+\theta_{i})H_{i}h_{i}$ and $(1+\theta_{i})H_{i}h_{i+1}$} are large then (\ref{eq:coeff3}) and (\ref{eq:coeff4}) are significantly more restrictive. Such a case occurs for problems with a small contact resistance ($1/H_{i}$) at the $i$th interface. Even though they are equivalent in the limit that $H_i \rightarrow \infty$, approximating Type GI conditions by Type GII conditions \change{by choosing a large value of $H_{i}$} is not recommended for the forward Euler scheme (\ref{eq:forward_euler}), as this leads to a very strict constraint on the time step.

\begin{figure}[htb]
\subfloat[Case E]{\includegraphics[width=0.5\textwidth]{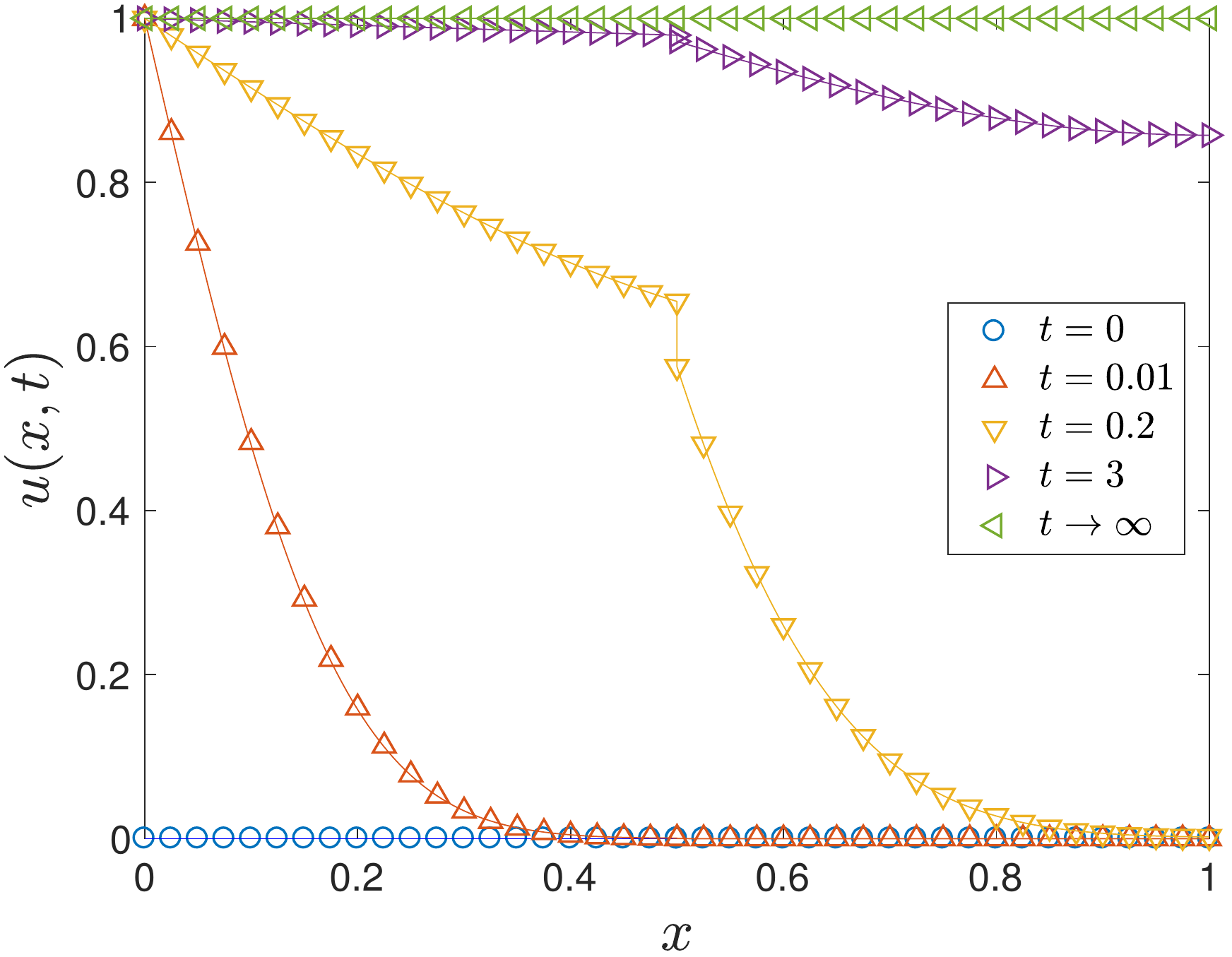}}
\subfloat[Case F]{\includegraphics[width=0.5\textwidth]{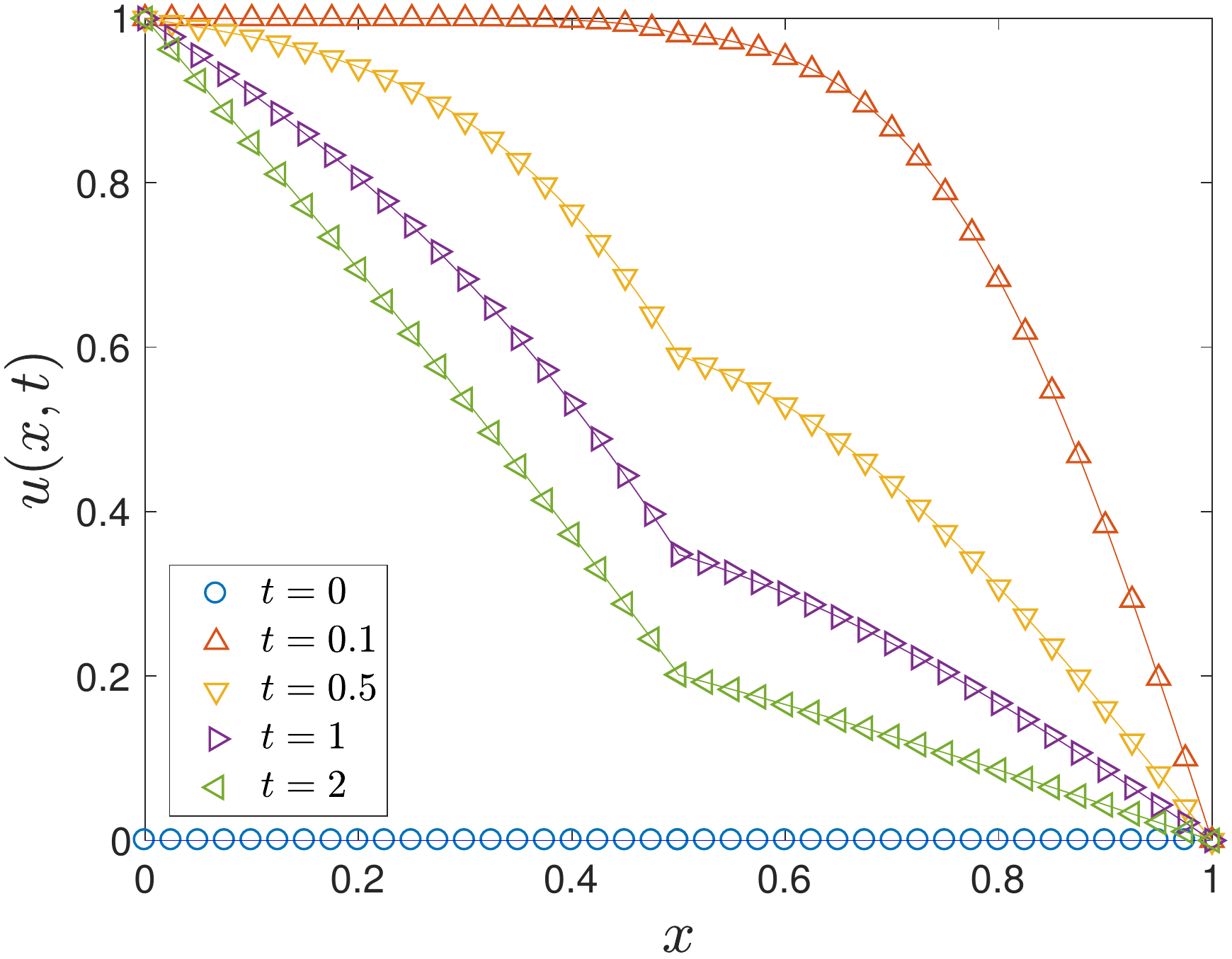}}
\centering
\caption{Numerical solution of Cases E and F using the forward Euler scheme (\ref{eq:forward_euler}) with node spacings $h_{1} = h_{2} = 0.025$ and the stable time steps of $\tau = 2.78 \times 10^{-4}$ and $\tau = 2.48\times 10^{-5}$, respectively.}
\label{fig:Stability}
\end{figure}

We conclude this section by presenting a problem (Case F) in which the stability condition arising from the interface conditions is approximately two orders of magnitude more restrictive than the classical stability condition (\ref{eq:tau_interior}). The problem consists of $m = 2$ layers with diffusivities $D_1 = 0.1$ and $D_2 = 0.2$, conductivities $\gamma_1 = 10^{-4}$ and $\gamma_2 = 5 \times 10^{-4}$ and node spacings $h_1 = h_2 = 0.025$. The domain under consideration is $[l_0,l_1,l_2] = [0,0.5,1]$ with Dirichlet conditions imposed on both boundaries $a_L = 1$, $b_L = 0$, $c_L = 1$, $a_R = 1$, $b_R = 0$ and $c_R = 0$, and Type II conditions (\ref{eq:type2interface1})--(\ref{eq:type2interface2}), with contact transfer coefficient $H_1 = 0.5$, imposed at the interface. 

For this problem, taking the maximum time step satisfying the classical stability condition ({\ref{eq:tau_interior}) in both layers, that is $\tau = \min\{h_{1}^{2}/(2D_{1}),h_{2}^{2}/(2D_{2})\} = 1.5625\times 10^{-3}$ yields an unstable solution as the spectral radius $\rho(\mathbf{A}_F) = 87.146$. Conversely, the maximum time step satisfying the stability conditions (\ref{eq:coeff3}) and (\ref{eq:coeff4}), $\tau = 2.48\times 10^{-5}$, is almost 100 times smaller and yields a stable solution (given in Figure \ref{fig:Stability}b) as $\rho(\mathbf{A}_F) = 0.9996$.

\section{Conclusions and summary}
\label{sec:conclusion}
This paper has developed a new finite volume method for the one-dimensional multilayer diffusion problem, capable of treating problems with general boundary/interface conditions. The new method is second-order accurate in space and, unlike existing schemes in the literature, preserves the tridiagonal matrix structure of the classical single-layer discretisation. Stability and convergence analysis of the method was presented for the three classical time discretisation schemes: forward Euler, backward Euler and Crank-Nicolson. \change{Notably, we demonstrated that the backward Euler and Crank-Nicolson schemes are always unconditionally stable. We also found that for the forward Euler scheme} certain types of interface conditions can lead to more restrictive stability conditions than simply applying the \change{classical stability} condition $\tau\leq h_{i}^{2}/(2D_{i})$ in each layer (where $D_{i}$ is the diffusivity and $h_{i}$ is the grid spacing in the $i$th layer). \change{In particular, we showed that Type GI interface conditions (\ref{eq:general_perfect1})--(\ref{eq:general_perfect2}) may lead to a more restrictive stability condition if $\theta_{i}\neq 1$ and Type GII interface conditions (\ref{eq:general_imperfect1})--(\ref{eq:general_imperfect2}) always lead to a more restrictive stability condition.} Numerical experiments confirmed the \change{second-order} spatial accuracy of the new finite volume schemes and confirmed the stability analysis. 

\revision{In this paper, we have derived finite volume schemes for linear multilayer diffusion processes without a source term. Modification of these schemes for treating multilayer reaction-diffusion problems with nonlinear reaction terms can be carried out in a straightforward manner. Moreover, when using an explicit treatment of the reaction term \cite{ruuth_1995}, it may be possible to extend our stability conditions by adapting analysis presented for the single-layer reaction-diffusion problem \cite{cherruault_1990}.}
  
Finally, in the future, we plan to investigate the extension of the finite volume method presented in this paper to higher dimensions.

\section*{Acknowledgements}
\noindent EJC acknowledges funding from the Australian Research Council (DE150101137). {Both authors acknowledge helpful discussions with Ian Turner (QUT) on spectral properties of tridiagonal matrices.}

\bibliographystyle{plainnat}
\bibliography{references}

\end{document}